\documentclass[a4paper,11pt]{article}

\usepackage[latin1]{inputenc}
\usepackage[T1]{fontenc}
\usepackage[british,UKenglish,USenglish,american]{babel}
\usepackage{amsmath}
\usepackage{amsfonts}
\usepackage{hyperref}
\usepackage[T1]{fontenc}
\usepackage[latin1]{inputenc}
\usepackage{theorem}
\usepackage{graphics}
\usepackage{makeidx}
\usepackage{fancyhdr}
\usepackage{hyperref}
\usepackage{verbatim}
\usepackage{bbm}
\usepackage{bm}
\usepackage{amssymb}
\usepackage{color}
\numberwithin{equation}{section}
\usepackage{amssymb}

\newcommand{\tim}[1]{{#1}}

\newcommand{\be}{\begin{equation}}
\newcommand{\ee}{\end{equation}} 
\newcommand{\benn}{\begin{equation*}}
\newcommand{\eenn}{\end{equation*}}
\newcommand{\bea}{\begin{eqnarray}}
\newcommand{\eea}{\end{eqnarray}}

\newcommand{\beann}{\begin{eqnarray*}}
\newcommand{\eeann}{\end{eqnarray*}}

\newtheorem{theorem}{Theorem}[section]

\newtheorem{corollary}[theorem]{Corollary}
\newtheorem{lemma}[theorem]{Lemma}
\newcommand{\xx}{\mathbb{X}}

\newtheorem{definition}[theorem]{Definition}
\theorembodyfont{\rm}

\newtheorem{remark}[theorem]{Remark}
\newtheorem{example}[theorem]{Example}
\newtheorem{assumptions}[theorem]{Assumptions}

\newcommand{\qed}{\hfill $\Box$\smallskip}

\newcommand{\X}{\mathbf{X}}
\newcommand{\XX}{\mathbb{X}}

\def\R{\mathbb{R}}

\def\N{\mathbb{N}}

\def\cA{\mathcal{A}}
\def\cB{\mathcal{B}}
\def\cC{\mathcal{C}}
\def\cD{\mathcal{D}}

\def\cF{\mathcal{F}}

\def\cI{\mathcal{I}}

\def\cL{\mathcal{L}}

\def\cO{\mathcal{O}}
\def\cP{\mathcal{P}}

\def\cR{\mathcal{R}}

\def\cX{\mathcal{X}}

\def\txtd{{\textnormal{d}}}

\def\Id{{\textnormal{Id}}}
\newcommand{\gubnorm}[2]{\left\|#1,#2 \right\|_{X,2\gamma,\alpha}}
\newcommand{\gubnormpar}[3]{\left\|#1,#2 \right\|_{X,2\gamma,#3}}
\newcommand{\norm}[1]{\left\lVert #1 \right\rVert}
\newcommand{\abs}[1]{\left| #1 \right|}

\newcommand*\samethanks[1][\value{footnote}]{\footnotemark[#1]}
\allowdisplaybreaks

\title{Stochastic evolution equations with rough boundary noise}

\author{Alexandra Neam\c tu\thanks{University of Konstanz, Department of Mathematics and Statistics,  Universit\"atsstra\ss{}e~10 78464 Konstanz, Germany. E-Mail: alexandra.neamtu@uni-konstanz.de, tim.seitz@uni-konstanz.de}
~~and~~Tim Seitz\samethanks
}
\begin{document}
    
    \maketitle
    
    \begin{abstract}
        We investigate the pathwise well-posedness of stochastic partial differential equations perturbed by multiplicative Neumann boundary noise, such as fractional Brownian motion for $H\in(1/3,1/2]$. Combining functional analytic tools with the controlled rough path approach, we establish global existence of solutions and flows for such equations. For Dirichlet boundary noise we obtain similar results for smoother noise, i.e.~in the Young regime.
    \end{abstract}
    \section{Introduction}
        We investigate the semilinear parabolic evolution equation with nonlinear rough boundary noise given by
        \begin{align}\label{eq:intro}
            \begin{cases}
                \frac{\partial}{\partial t} y = \cA y +  f(y) & \text{ in } \cO,\\
                \cC y = F(y)~\frac{\txtd}{\txtd t} \X & \text{ on } \partial \cO,\\
                y(0)=y_0.
            \end{cases}
        \end{align}
        where $\cO\subset \R^d$ is a bounded domain with $C^\infty$-boundary, $\X$ is a $\gamma$-H\"older rough path with $\gamma\in(\frac{1}{3},\frac{1}{2}]$, $\cA$ is a second order differential operator in divergence form with Neumann boundary conditions $\cC$ as specified in \eqref{formalOp}, $y_0$ is the initial data and $f$ and $F$ are nonlinear terms. Consequently, the theory developed in this article is applicable to the fractional Brownian motion with Hurst parameter $H\in(\frac{1}{3},\frac{1}{2}]$. For $H=\frac{1}{2}$ we recover the results for the Brownian motion in \cite{VeraarSchnaubelt}.
        Setting for now the drift term $f=0$, we note that the solution of~\eqref{eq:intro} should be given by
        \begin{align}\label{eq:RoughConv}
            y_t = S_t y_0+ A\int_0^t S_{t-r} NF(y_r)~\txtd \X_r,
        \end{align}
        where $A$ will be the $L^p$-realization of $\mathcal{A}$ which generates an analytic semigroup $(S_t)_{t\in[0,\infty)}$ on a suitable scale of Banach spaces. Furthermore, $N$ is the Neumann operator, which maps the boundary data into the interior of the domain. 
        The derivation of~\eqref{eq:RoughConv} was established in~\cite{daprato,m} for additive noise.
        One of the main goals of this work is to give a meaning to the convolution~\eqref{eq:RoughConv} using a controlled rough path approach and construct pathwise solutions for~\eqref{eq:intro}. 
        Our approach is applicable only to Neumann boundary noise for $\gamma\in(\frac{1}{3},\frac{1}{2}]$. The Dirichlet case can be included in the Young regime, provided that the H\"older regularity of the noise satisfies an additional condition. This fact is not  surprising, since it is well-known that mild solutions for stochastic evolution equations with Brownian Dirichlet boundary noise generally fail to exist~\cite{daprato,VeraarSchnaubelt}.  Using different techniques and tools from Malliavin calculus, we mention that Dirichlet boundary noise can be treated for a Brownian  motion, see \cite{alos}.
        However, one can construct solutions for Dirichlet boundary noise, if the random input is given by a fractional Brownian motion where the Hurst index $H>\frac{3}{4}$~\cite{DuncanMaslowski}. Furthermore, one can show well-posedness of~\eqref{eq:intro} driven by an additive fractional Brownian motion if  $H\in(\frac{1}{4},1)$ as obtained in~\cite[Section 5]{DuncanMaslowski} and~\cite{DuncanMaslowski2}.
        We recover these thresholds for the Hurst parameters in the more general framework of rough path theory, which allows us to consider multiplicative noise in contrast to~\cite{CMO,debussche,DuncanMaslowski, DuncanMaslowski2,m}. Results on stochastic evolution equations of the type~\eqref{eq:intro} with multiplicative noise can be looked up in~\cite{Maslowski,VeraarSchnaubelt}. Here, the random input $\X$ is an infinite-dimensional Brownian motion. To our best knowledge, for fractional Brownian motion, only the additive case was considered in~\cite{CMO,DuncanMaslowski, DuncanMaslowski2}. Results for SPDEs with boundary L\'evy noise have been obtained in~\cite{BP,HP}.\\
        On the other hand, there has been a growing interest in developing a solution theory based on a semigroup approach for rough evolution equations starting with~\cite{Gubinelli,GubinelliTindel} and the more recent approaches~\cite{FrizHairer,GHairer, GHN, HN19}. However, most of the results are stated on a torus and \tim{seem to be only} applicable to rough PDEs with zero Dirichlet or Neumann boundary conditions. To our best knowledge, there are no works that deal with more complicated boundary conditions (nonlinear or dynamic) or with boundary noise using rough paths techniques. Here we contribute to this aspect and provide a solution theory for~\eqref{eq:intro} with Neumann boundary noise. Our approach relies on controlled rough paths on a monotone scale of interpolation spaces introduced in~\cite{GHN}. One major difficulty in our setting is that we have to deal with controlled rough paths belonging to two different scales of Banach spaces:~one for the solution and one for the boundary data.~However, the main technical challenge is to obtain enough spatial regularity  in order to guarantee that~\eqref{eq:RoughConv} belongs to the domain of $A$ and to investigate its Gubinelli derivative.
It is well-known that~\cite{GHairer,GHN,GubinelliTindel, HN19} there is a trade-off between space and time regularity required in order to define the standard convolution $\int_0^t S_{t-r}F(y_r)~\txtd{\bf X}_r$ using pathwise arguments. In particular the nonlinear term $F$ has to improve the spatial regularity as in~\cite{HN19,Maria} or is allowed to lose spatial regularity which is strictly less than the time regularity of the noise~\cite{GHairer,GHN}. In particular for the Brownian motion, which is $1/2-$ H\"older regular, 
this means that
it is not yet possible to deal with transport-type noise in the mild formulation~\cite{GHairer, GHN} as opposed to the It\^o calculus.~However, there are numerous advantages to consider a rough path formulation for~\eqref{eq:intro} such as a continuous dependence of the solution with respect to the random input and the existence of random dynamical systems, which we discuss in Section~\ref{sec:rds}.~In our setting, we have to incorporate the Neumann operator in the rough convolution, compute a Gubinelli derivative and make sure that~\eqref{eq:RoughConv} belongs to the domain of $A$. 
We can deal with these issues by introducing a suitable extrapolation operator which entails a convenient representation of the mild solution.\\
The idea of working with extrapolation operators in the context of rough path theory is new, and the results obtained are of independent interest. In the works~\cite{DuncanMaslowski, DuncanMaslowski2} due to the presence of an additive fractional noise on the boundary, the application of extrapolation operators is not required. In order to deal with multiplicative Brownian noise on the boundary, extrapolation operators have been introduced in~\cite{VeraarSchnaubelt}. Similar to~\cite{VeraarSchnaubelt}, the main idea is to rewrite~\eqref{eq:intro} as a semilinear evolution equation without boundary noise using extrapolation operators. The novel aspect of this work is to analyze the interaction between extrapolation operators, controlled rough paths and rough integrals which is necessary in order to investigate the well-poosedness of~\eqref{eq:intro}. \\
        Finally, we mention that the motivation of incorporating boundary noise arises in the study of transport models under random sources~\cite{BDS,Wang}. In this framework, the recent work~\cite{ns} considers a variant of the 3D Navier-Stokes equations subject to stochastic wind driven boundary conditions which are modelled by an additive cylindrical Brownian motion. 
        Furthermore, beyond the well-posedness theory mentioned above, optimal control results for SPDEs with boundary noise have been obtained in~\cite{bt,debussche}, whereas numerical aspects have been discussed in~\cite{dirk}. Averaging principles and fluctuations around the averaged equation for such equations have been derived in~\cite{Cerrai}.\\
        
        This work is structured as follows.~In Section~\ref{preliminaries} we collect fundamental results from the controlled rough path approach according to a monotone interpolation family as introduced in~\cite{GHN}. 
        In our case, we have to work with two different scales of interpolation spaces in order to deal with the boundary data. More precisely, the boundary data belongs to the Besov scale, whereas the solution is expected to belong to the Bessel potential scale according to the boundary conditions $\cC u =0$. We specify these function spaces in Section~\ref{preliminaries} and further provide a background on extrapolation operators. Section~\ref{main} contains the main results of this work. We analyse the interplay between the extrapolation operators and controlled rough paths. Based on this, we construct pathwise (local- and global-in-time) solutions for~\eqref{eq:intro} in Theorem~\ref{local} and Theorem~\ref{thm:global}. Moreover, we also establish in Theorem~\ref{globalDirichlet} a well-posedness result for Dirichlet boundary noise if $H\in(\frac{3}{4},1)$ using Young's integral.
        Section~\ref{sec:rds} contains a direct application of our global well-posedness result. Namely, we establish in Theorem~\ref{rds} the existence of random dynamical systems associated to~\eqref{eq:intro}. Its long-time behavior will be addressed in a future work. For example, it is known that white noise on the boundary can have a stabilization effect~\cite{Stefanie, Munteanu}. However, such results are not available for fractional noise in the context of the rough path approach. 
        Finally, we conclude with some applications of our theory in Section~\ref{examples}. \\
        
        {\bf Acknowledgements.} We thank Robert Denk for very helpful discussions on extrapolation spaces. We are grateful to the referees for the numerous valuable comments.

    \section{Preliminaries}\label{preliminaries}
        In this section we collect  basic results on rough paths and rough convolutions for semilinear parabolic problems~\cite{GHN} and 
        provide the assumptions on the coefficients of~\eqref{eq:intro}. Furthermore, we explain the concept of extrapolation spaces and operators.
        %One of the main difficulties here is that we have to deal with two different scales of Banach spaces. More precisely, as justified in Section~\ref{main}, the boundary data will belong to some Besov space, whereas the solution of~\eqref{eq:intro} is expected to belong to the Bessel potential scale. \\
        
        We first specify the type of noise we consider. The random input is a $d$-dimensional $\gamma$-H\"older rough path   $\textbf{X}:=(X,\mathbb{X})$, for $\gamma\in(1/3,1/2]$ and $X_0=0$. Here we assume without loss of generality that $d=1$, since the generalization to $d>1$ can be made componentwise.
        More precisely, we have
        \begin{align*}
            X\in C^{\gamma}([0,T];\mathbb{R}) ~~\mbox{ and } ~~ \mathbb{X}\in C^{2\gamma}([0,T]^2;\mathbb{R}\otimes\mathbb{R})
        \end{align*}
        and the connection between $X$ and $\mathbb{X}$ is given by Chen's relation
        \begin{align}\label{chen}
            \mathbb{X}_{t,s}- \mathbb{X}_{u,s}-\mathbb{X}_{t,u}=X_{u,s}\otimes X_{t,u},
        \end{align}
        where we write $X_{u,s}:=X_u-X_s$ for any path.
        The term $\mathbb{X}$ is sometimes referred to in the literature as a second order process. We further introduce an appropriate distance between two $\gamma$-H\"older rough paths.
        \begin{definition}
            Let $J\subset\mathbb{R}$ be a compact interval, $\Delta_J:=\{ (s,t)\in J \times J : s\leq t \}$ and $\mathbf{X}=(X,\mathbb{X})$ and $\mathbf{\tilde{X}}=(\tilde{X},\tilde{\mathbb{X}})$ be two $\gamma$-H\"older rough paths. We introduce the $\gamma$-H\"older rough path (inhomogeneous) metric
            \begin{align}\label{rp:metric}
                d_{\gamma,J}(\mathbf{X},\mathbf{\tilde{X}} )
                := \sup\limits_{(s,t)\in \Delta_J} \frac{|X_{t,s}-\tilde{X}_{t,s}|}{|t-s|^{\gamma}}
                + \sup\limits_{(s,t) \in \Delta_{J}}
                \frac{|\mathbb{X}_{t,s}-\tilde{\mathbb{X}}_{t,s}|} {|t-s|^{2\gamma}}.
            \end{align}
            We set $\rho_\gamma(\mathbf{X}):=d_{\gamma,[0,T]}(\mathbf{X},0)$.
        \end{definition}
        We specify the necessary assumptions on the linear part of~\eqref{eq:intro}. For concrete examples of operators satisfying such properties, see Section~\ref{examples}.~Let $A$ be the $L^p(\cO)$-realization of $\mathcal{A}$ with respect to the Neumann boundary conditions.~Therefore, its domain is given by $D(A):=\{u\in H^{2,p}(\cO): \cC u=0\}$. \\
        
Throughout this manuscript we make the following assumptions. 
        \begin{assumptions}\label{ass}
            \begin{itemize}
                \item [1)] The boundary value problem $(\mathcal{A},\cC)$ is normally elliptic. 
                \item [2)] The operator $A$ generates an analytic semigroup $(S_t)_{t\in [0,\infty)}$. Furthermore, assume that $A$ has bounded imaginary powers, which means that there exists a constant $C > 0$ such that 
                \begin{align*}
                    \norm{A^{is}}_{\cL(L^p(\cO))} \leq C
                \end{align*}
            for all $s\in \R$ with $\abs{s}\leq 1$.
            \end{itemize} 
        \end{assumptions}
        
        In order to deal with the boundary data of~\eqref{eq:intro} we further need to introduce the solution operator for an abstract boundary value problem. To this aim, we consider for $p\in [2,3]$ and $\alpha>1+\frac{1}{p}$ the normally elliptic boundary value problem
        \begin{align}\label{eq:bvp}
            \begin{cases}
            \cA u = 0 & \text{ in } \cO,\\
            \cC u = g & \text{ on } \partial \cO.
            \end{cases}
        \end{align}
        It is well-known, see for example \cite[Theorem 9.2]{Amann}, that the problem \eqref{eq:bvp} has for $g\in B^{\alpha-1-\frac{1}{p}}_{p,p}(\partial \cO)$, a unique strong solution in $H^{\alpha,p}(\cO)$. The corresponding solution operator $$N:B^{\alpha-1-\frac{1}{p}}_{p,p}(\partial\cO)\to H^{\alpha,p}(\cO)$$ is also called the Neumann operator. %\tim{The function spaces which appear in this setting, besides the $L^p$-spaces, are the Besov $B^s_{p,q}(\cO)$ and Bessel potential spaces $H^{s,p}(\cO)$. Here we will not go into the definition of the spaces, for detailed information see for example in \cite[Section 4.2]{Triebel78}.}
        In particular, the Neumann operator is bounded between these two function spaces. Here, we immediately see that we need two scales of Banach spaces (as specified in Definition~\ref{raum}) for the solution defined on the domain $\cO$ and for the boundary data (which lies on $\partial\cO$), in order to investigate~\eqref{eq:intro}. These are given by the Besov spaces $B^{\alpha-1-\frac{1}{p}}_{p,p}(\partial\cO)$ and the Bessel potential scale $H^{\alpha,p}(\cO)$ generated by the operator $A$, see Definition \ref{def:banachScale}. In the following we introduce the necessary properties of these space required in this work and refer to \cite[Section 4.2]{Triebel78} for further details.
        \begin{definition}\label{raum}
            A family of separable Banach spaces $(\cB_\alpha,\abs{\cdot}_{\cB_{\alpha}})_{\alpha\in\mathbb{R}}$ is called a monotone family of interpolation spaces if for $\alpha_1\leq \alpha_2$, there is a dense and continuous embedding $\cB_{\alpha_2}\hookrightarrow \cB_{\alpha_1}$  and the following interpolation inequality holds for $\alpha_1\leq \alpha_2\leq \alpha_3$ and $x\in  \cB_{\alpha_3}$:
            \begin{align}\label{interpolation:ineq}
                |x|^{\alpha_3-\alpha_1}_{\cB_{\alpha_2}} \lesssim |x|^{\alpha_3-\alpha_2}_{\cB_{\alpha_1}} |x|^{\alpha_2-\alpha_1}_{\cB_{\alpha_3}}.
            \end{align}
        \end{definition}
        \begin{remark}
            For example, the scale $(B^\alpha_{p,p}(\partial \cO))_{\alpha\in \R}$ satisfies Definition \ref{raum}. The interpolation inequality follows from the fact that 
            \begin{align*}
                \left[B^{\alpha_1}_{p,p}(\partial \cO),B^{\alpha_2}_{p,p}(\partial \cO)\right]_{\vartheta}=B^\alpha_{p,p}(\partial \cO)
            \end{align*}
            for $\alpha_1< \alpha_2$ and $\alpha=(1-\vartheta)\alpha_1+\vartheta \alpha_2$ for $\vartheta\in (0,1)$. Since complex interpolation is exact, the inequality \eqref{interpolation:ineq} is satisfied. A further example is given by the Bessel potential scale, which is introduced in Section~\ref{main}.
        \end{remark}
        The main advantage of this approach is that we can view the semigroup $(S_t)_{t\in [0,\infty)}$ generated by $A$ as a linear mapping between these interpolation spaces and obtain the following standard bounds for the corresponding operator norms. If $S:[0,T]\to \cL(\cB_{\alpha},\cB_{ \alpha+1})$ is such that for every $x\in \cB_{\alpha+1}$ and $t\in(0,T]$ we have that $|(S_t-\Id)x|_{\cB_{\alpha}} \lesssim t|x|_{\cB_{\alpha+1}}$ and $|S_tx|_{\cB_{\alpha+1}}\lesssim t^{-1} |x|_{\cB_{\alpha}}$, then for every $\sigma\in[0,1]$ we have that $S_t\in\cL(\cB_{\alpha+\sigma})$ and 
        \begin{align}
            |(S_t-\Id) x|_{\cB_{\alpha}}&\lesssim t^\sigma |x|_{\cB_{\alpha+\sigma}},\label{hg:1}\\
            |S_tx|_{\cB_{\alpha+\sigma}}&\lesssim t^{-\sigma}|x|_{\cB_{\alpha}}\label{hg:2}.
        \end{align}
        Now we introduce the notion of Banach scales and within extrapolated spaces and operators. 
        \begin{definition}{\em (\cite[Section V.1.1]{Amann2})}\label{def:banachScale}
            Let $J$ be an index set \tim{such that for any $\alpha\in J$, $\alpha+1\in J$}. We call the couple $(\cB_\alpha,A_\alpha)_{\alpha\in J} $ a Banach scale, if
            \begin{itemize}
                \item $(\cB_\alpha,\abs{\cdot}_{\tim{\cB_\alpha}})$ is a Banach space for every $\alpha\in J$ and $\cB_{\alpha_1}\hookrightarrow \cB_{\alpha_2}$ for $\alpha_1>\alpha_2$. We denote the embedding with $j_{\alpha_1}^{\alpha_2}$.
                \item $A_\alpha\in \cL (\cB_{\alpha+1},\cB_\alpha)$ is an isomorphism for every $\alpha\in J$.
                \item For every $\alpha_1>\alpha_2$ we have the equality
                \begin{align*}
                j_{\alpha_1}^{\alpha_2}A_{\alpha_1}=A_{\alpha_2}j_{\alpha_1+1}^{\alpha_2+1}.
            \end{align*}
        \end{itemize}
        If these embeddings are dense, we call the Banach scale densely injected.
        \end{definition}
        \begin{remark}\label{rem:Extrapolation}
            %\begin{itemize}
                %\item[(i)] 
                As a direct consequence of this definition, we know that for $\alpha_1>\alpha_2$, $A_{\alpha_1}$ is the $\cB_{\alpha_1}$-realization of $A_{\alpha_2}$. Furthermore, if the scale is densely injected, all operators are completely determined by $A_0$. Based on this, it is possible to construct the scale out of a single operator, \cite[Remark 1.1.2]{Amann2}. 
                %\item[(ii)] 
                One of the main advantages we get from this fact is that $A_{\alpha_1}\subset A_{\alpha_2}$ holds for all $\alpha_2<\alpha_1$. Therefore, for every $x\in \cB_{1+\alpha_1}$ we have the equality $A_{\alpha_1}x = A_{\alpha_2} x$. We frequently use this property  throughout this manuscript.
           % \end{itemize}
        \end{remark}
        In our case, we consider an operator $A=:A_0$ satisfying Assumption \ref{ass} and set $\cB_0:=L^p(\cO),~ \cB_1:=D(A)$. From this starting point, we build a Banach scale, which is  uniquely determined by $A$. For $\alpha >0$, we use the fractional powers of $A$ as follows.~\tim{We define for $\alpha>0$ and $\lambda$ in the resolvent of $A$, the spaces
        \begin{align*}
            \cB_\alpha:=D((A-\lambda \text{Id})^\alpha),
        \end{align*}
        equipped with the norm 
$\abs{\cdot}_{\cB_\alpha}=\norm{(A-\lambda \text{Id}))^\alpha \cdot}_{L^p(\cO)}$
        %$\abs{\cdot}_{\cB_\alpha}:=\norm{A^\alpha\cdot}_{L^p(\cO)}$
        and let $A_\alpha$ be the $\cB_\alpha$-realization of $A_0$. Furthermore $\lambda=0$ can be assumed without loss of generality by shifting the operator $A$.}
        
        It can be shown that $(\cB_\alpha, A_\alpha)_{\alpha\in [0,\infty)}$ forms a (one-sided) densely injected Banach scale, see \cite[Theorem 1.2.4]{Amann2}. \\ 
        To extend this to negative indices, we use the theory of extrapolation spaces. The idea behind this concept comes from the fact, that $\cB_0$ can be reconstructed from $\cB_1$. To see this, note that for $x\in \cB_1$ we have $\abs{A_1^{-1} x }_{\cB_1}=\abs{x}_{\cB_0}$. Then $\cB_0$ is the completion of $\cB_1$ equipped with $\abs{A_1^{-1}\cdot}_{\cB_1}$. This now motivates the definition of extrapolation spaces as a super space of $\cB_0$ similarly constructed.  We define $\cB_{-1}$ as the completion of $\cB_0$ with respect to the norm $\abs{A_0^{-1}\cdot}_{\cB_0}$ and let $A_{-1}$ be the closure of $A_0$ in $\cB_{-1}$. For $\alpha \in (0,1)$ we define $\cB_{-\alpha}:=\left[\cB_{-1},\cB_0\right]_{1-\alpha}$ and let $A_{-\alpha}$ be the $\cB_{-\alpha}$-realization of $A_{-1}$. The operators with negative indices are  also called {\em extrapolated operators}. 
        \begin{remark}
            The same procedure can be done iteratively to extend the scale up to the index set $[-m,\infty)$ for an arbitrary $m\in \N$. Note that in every new step, we have to replace the previously constructed extrapolated spaces $\cB_{-m+1}$ by the isomorphic image $j_{-m+1}^{-m+2}(\cB_{-m+2})$ in order to ensure the validity of the dense embeddings of the form $\cB_{-m+2}\hookrightarrow \cB_{-m+1}\hookrightarrow \cB_{-m}$, see \cite[Section V.1.3]{Amann2}.
        \end{remark}
        \begin{theorem}{\em(\cite[Theorem V.1.2.1, Theorem V.1.5.4]{Amann2})}\label{thm:EpolOperator}
            The scale $(A_\alpha,\cB_\alpha)_{\alpha\in [-m,\infty)}$ constructed above forms a densely injected Banach scale in the sense of Definition \ref{def:banachScale}. Furthermore, for $-m\leq \alpha_1<\alpha_2<\infty$ and $\vartheta \in [0,1]$ we have the reiteration property
            \begin{align*}
                \left[\cB_{\alpha_1},\cB_{\alpha_2}\right]_\vartheta=\cB_{(1-\vartheta)\alpha_1+\vartheta \alpha_2},
            \end{align*}
            where $[\cdot,\cdot]_\vartheta$ stands for  complex interpolation.
            Moreover, $(\cB_\alpha)_{\alpha\in[-m,\infty)}$ satisfies the interpolation inequality \eqref{interpolation:ineq} and $(A_\alpha)^{\alpha-\beta}:\cB_\alpha\to \cB_\beta$ is an isomorphism for $-m\leq \alpha<\beta<\infty$.
        \end{theorem}
        Based on this result, we conclude that the space part of the scale introduced above forms a family of monotone interpolation families as specified in Definition \ref{raum}.
        We will refer to it as the interpolation-extrapolation scale generated by $A$. Moreover, we are interested in semigroups generated by the operators in a Banach scale. In this case, a similar statement to Remark \ref{rem:Extrapolation} i) holds true. 
        \begin{theorem}{\em(\cite[Proposition V.1.5.5]{Amann2},~\cite[Theorem V.2.1.3]{Amann2})} \label{thm:EpolSemigroup}
            Let $A$ satisfy the Assumption \eqref{ass} and consider the interpolation-extrapolation scale generated by $A$.~Then for every $\alpha\in[-m,\infty)$ the operator $A_\alpha$ has bounded imaginary powers and generates an analytic semigroup. Further, for $\alpha>\beta$ we have that the semigroup generated by $A_\alpha$ is the same as the semigroup generated by $A_\beta$ restricted to $\cB_\alpha$.
        \end{theorem}
        Further details on the theory of Banach scales or extrapolation can be found in \cite[Section 6--7]{Amann} and \cite[Chapter V]{Amann2}.\\
        
        Keeping this in mind, we now introduce the following definition of a controlled rough path tailored to the parabolic structure of the PDE we consider, in the spirit of~\cite{GHN}. This is convenient for our aims, since the semigroup will not be incorporated in the definition of the controlled rough path as in~\cite{GHairer} or alternative approaches~\cite{GubinelliTindel, HN19} which iterate the stochastic convolution into itself. 
        
        \begin{definition} {\em (Controlled rough path according to a monotone family $(\cB_\alpha)_{\alpha\in \R}$).}\label{def:crp}
            We call a pair $(y,y')$ a controlled rough path for a fixed $\alpha\in \R$ if
            \begin{itemize}
                \item $(y,y')\in C([0,T];\cB_\alpha) \times ((C[0,T];\cB_{\alpha-\gamma} ) \cap C^{\gamma}([0,T];\cB_{\alpha-2\gamma} ))$. The component $y'$ is referred to as Gubinelli derivative of $y$.
                \item the remainder  \begin{align}\label{remainder}
                R^y_{t,s}= y_{t,s} -y'_s X_{t,s}    
            \end{align}
            belongs to $ C^{\gamma}([0,T]^2;\cB_{\alpha-\gamma})\cap C^{2\gamma}([0,T]^2;\cB_{\alpha-2\gamma})$.
            \end{itemize}
            \tim{The space of controlled rough paths is denoted by $\cD^{2\gamma}_{X,\alpha}$ and endowed with the norm $\norm{\cdot}_{X,2\gamma,\alpha}$ given by~\cite{GHN}
            \begin{align}\label{g:norm}
                \gubnorm{y}{y'}:= \left\|y \right\|_{\infty,\cB_\alpha} 
                + \|y' \|_{\infty,\cB_{\alpha-\gamma}}
                + \left[y'\right]_{\gamma,\cB_{\alpha-2\gamma}}
                %+ \left\|R^y \right\|_{\alpha,\cB_{\gamma-\alpha}}%
            +\left[R^y \right]_{\gamma,\cB_{\alpha-\gamma}}    + \left[R^y \right]_{2\gamma,\cB_{\alpha-2\gamma}}.
            \end{align}}
        \end{definition}
        Note that for paths $h:[0,T]\to \cB$ and second order processes $g:[0,T]^2\to \cB$ with values in a Banach space $\cB$
        \begin{align*}
            \left[h\right]_{\gamma,\cB}:= \sup\limits_{(s,t)\in \Delta_{[0,T]}} \frac{\abs{h_t-h_s}_\cB}{\abs{t-s}^\gamma}, \quad \left[g\right]_{\gamma,\cB}:= \sup\limits_{(s,t)\in \Delta_{[0,T]}} \frac{\abs{g_{t,s}}_\cB}{\abs{t-s}^\gamma}
        \end{align*}
        denotes the $\gamma$-H\"older seminorm. 
        %\textcolor{blue}{Mention why $\left[R^y \right]_{\alpha,\cB_{\gamma-\alpha}}$ is no part of the norm? I added this for simplicity to stick to~\cite{GHN}. We show in the quasilinear paper that they are equivalent but we need to introduce some assumptions and more concepts and this is not important here}
        The first index in the notation above always indicates the time regularity, and the second one stands for the space regularity. For simplicity, we write  $|y|_{\alpha}:=|y|_{\cB_\alpha}$,  $\|y\|_{\infty,\alpha}:=\|y\|_{\infty,\cB_\alpha}$ and $[y']_{\gamma,\alpha-2\gamma}:=[y']_{\gamma,\cB_{\alpha-2\gamma}}$ and analogously for the remainder. If we deal with different scales, we  write the full subscript.
        In order to emphasize the time horizon, we write $\cD^{2\gamma}_{X,\alpha}([0,T])$ instead of $\cD^{2\gamma}_{X,\alpha}$.  Furthermore, when the time interval is clear from the context, we  use the abbreviation $C^{\gamma'}(\cB_{\alpha'})$ for suitable $\gamma'$ and $\alpha'$ to point out the interplay between space and time regularity.  
        \begin{remark}\label{rem:CRPy}
            Note that we do not make the H\"older continuity of $y$ part of the definition of a controlled rough path, since
            using~\eqref{remainder} one immediately obtains for $\theta \in \left\{\gamma,2\gamma \right\}$ that
            \begin{align}\label{est:hoelder:y}
            \left[y \right]_{\gamma,\alpha-\theta}
            \leq \left\|y' \right\|_{\infty,\alpha-\theta} \left[X\right]_{\gamma} + \left[R^y \right]_{\gamma,\alpha-\theta}.
            \end{align}
        \end{remark}
        
        Given a controlled rough path, one can introduce the rough integral as follows~\cite[Theorem 4.5]{GHN}.
        
        \begin{theorem}
            \label{integral} Let $(y,y')\in \cD^{2\gamma}_{X,\alpha}$. Then 
            \begin{align}\label{Gintegral}
                \int\limits_{s}^{t} S_{t-r}y_{r}~\txtd \textbf{X}_{r} :=\lim\limits_{|\mathcal{P}|\to 0} \sum\limits_{[u,v]\in\mathcal{P}} S_{t-u}y_{u}X_{v,u} + S_{t-u}y'_{u}\mathbb{X}_{v,u},
            \end{align}
            where $\mathcal{P}$ denotes a partition of $[s,t]$ and the limit exists as an element in $\cB_{\alpha-2\gamma}$. For $0\leq \beta< 3\gamma$ the following estimate
            \begin{align}\label{estimate:integral}
                \abs{ \int\limits_{s}^{t} S_{t-r} y_{r}~\txtd\textbf{X}_{r} - S_{t-s}y_{s}X_{t,s} -S_{t-s}y'_{s}\mathbb{X}_{t,s}}_{\alpha-2\gamma+\beta} \lesssim \rho_\gamma(\X) \gubnorm{y}{y'} (t-s)^{3\gamma-\beta}
            \end{align}
            holds true. 
        \end{theorem}
        We emphasize that the stochastic convolution increases the spatial regularity of the controlled rough path,  see~\cite[Corollary 4.6]{GHN} and \cite[Lemma 3.5]{HN21}. We recall this result, which will be used later on.
        \begin{corollary}\label{cor:higherreg}
            Let $(y,y')\in\cD^{2\gamma}_{X,\alpha}([0,T])$, where $T\in[0,1]$ and $0\leq \sigma<\gamma$. Then the integral map
            \begin{align*}
                (y,y')\mapsto (z,z') := \Big( \int_0^\cdot S_{\cdot-r}y_r~\txtd \X_r, y_\cdot \Big)
            \end{align*}
            maps $\cD^{2\gamma}_{X,\alpha}([0,T])$ into itself. Moreover, we have a $C>0$ such that
            \begin{align}\label{est:Integral2}
                \gubnormpar{z}{z'}{\alpha+\sigma} \leq |y_0|_{\alpha} + |y'_0|_{\alpha-\gamma} + C T^{\gamma-\sigma} (1+\rho_\gamma(\X)) \gubnorm{y}{y'}.
            \end{align}
        \end{corollary}
        If the random input $X$ is more regular,~i.e. $\widetilde{\gamma}\in(\frac{1}2,1)$, then~\eqref{Gintegral}
        reduces to the {\em Young integral} which can be defined for a path $y\in\mathfrak{C}:= C(\cB_\alpha)\cap C^{\widetilde{\gamma}}(\cB_{\alpha-\widetilde{\gamma}})$ as
        \begin{align}\label{young}
            \int_s^t S_{t-r} y_r~\txtd X_r =\lim\limits_{|\cP|\to 0} \sum\limits_{[u,v]\in\cP} S_{t-u} y_u X_{v,u}, 
        \end{align}
        whereas~\eqref{estimate:integral} reads as
        \begin{align}\label{est:young}
            \abs{ \int\limits_{s}^{t} S_{t-r} y_{r}~\txtd X_{r} - S_{t-s}y_{s}X_{t,s}} _{\alpha-\widetilde{\gamma}+\beta} \lesssim \|y\|_{\mathfrak{C}} [X]_{\widetilde{\gamma}}  (t-s)^{2\widetilde{\gamma}-\beta},
        \end{align}
        for $\beta<2\widetilde{\gamma}$.
    \section{Main results}\label{main}
        We recall that $\gamma\in (\frac{1}{3},\frac{1}{2}]$ indicates the time regularity of the $\gamma$-H\"older rough path $\X:=(X,\XX)$.
        The main goal of this section is to prove that~\eqref{eq:RoughConv} is well-defined in the space of controlled rough paths.
        Since the boundary data of~\eqref{eq:intro} will belong to some Besov space, whereas the solution is expected to belong to a Bessel potential scale, see \eqref{eq:BanachScale}, we first fix two abstract scales of Banach spaces $(\cB_\alpha)_{\alpha\in\R}$ and $(\widetilde\cB_\alpha)_{\alpha\in\R}$. Furthermore, we denote the corresponding space of controlled rough paths by $\cD^{2\gamma}_{X,\alpha}:=\cD^{2\gamma}_{X}(\cB_\alpha)$ respectively $\widetilde\cD^{2\gamma}_{X,\alpha}:=\cD^{2\gamma}_X(\widetilde\cB_\alpha)$. For simplicity, we fix the time horizon $T\leq 1$ throughout this section.\\
        The first step is to define~\eqref{eq:RoughConv}
        using Theorem~\ref{integral}.
        Therefore, we begin this section by examining how a controlled rough path changes under the influence of the Neumann operator.
        \begin{lemma}\label{lem:CRP}
            Let $(\cB_{\alpha})_{\alpha\in \R}$ and $(\widetilde{\cB}_{\alpha})_{\alpha\in \R}$ be two monotone scales of interpolation spaces and $N\in \bigcap_{i=0}^2 \cL(\widetilde{\cB}_{\alpha_1-i\gamma},\cB_{\alpha_2-i\gamma})$ for $\alpha_1,\alpha_2\in \R$. Then we obtain for every $(y,y^\prime)\in \widetilde{\cD}^{2\gamma}_{X,\alpha_1}$ that $(Ny,Ny^\prime)\in \cD^{2\gamma}_{X,\alpha_2}$.
        \end{lemma}
        \begin{proof}
            As a direct consequence of the assumption on $N$ we see $(Ny,Ny^\prime)\in C(\cB_{\alpha_2})\times C(\cB_{\alpha_2-\gamma_2})$. Furthermore, we define the remainder by $R^{Ny}:=NR^y$ and get for $\vartheta\in \{\gamma,2\gamma\}$
            \begin{align*}
                [Ny^\prime]_{\gamma,\cB_{\alpha_2-2\gamma}}&\leq \norm{N}_{\cL(\widetilde{\cB}_{\alpha_1-2\gamma},\cB_{\alpha_2-2\gamma})}[y^\prime]_{\gamma,\widetilde{\cB}_{\alpha_1-2\gamma}},\\
                [R^{Ny}]_{\tim{\vartheta},\cB_{\alpha_2-\vartheta}}&\leq \norm{N}_{\cL(\widetilde{\cB}_{\alpha_1-\vartheta},\cB_{\alpha_2-\vartheta})}[R^y]_{\tim{\vartheta},\widetilde{\cB}_{\alpha_1-\vartheta}}.
            \end{align*}
            Consequently $Ny^\prime\in C^\gamma(\cB_{\alpha_2-2\gamma})$ and $R^{Ny}\in C^\gamma(\cB_{\alpha_2-\gamma})\cap C^{2\gamma}(\cB_{\alpha_2-2\gamma})$.~Since $Ny_{t,s}=Ny^\prime_{s}X_{t,s}+R^{Ny}_{t,s}$, this concludes the proof. \qed
        \end{proof}\\
        \\
        Regarding this, we can define the rough convolution based on Theorem~\ref{integral} as follows. 
        \begin{corollary}\label{cor:RoughConv}
            Let  $(\cB_{\alpha})_{\alpha\in \R}$, $(\widetilde{\cB}_{\alpha})_{\alpha\in \R}$ and $N$ as in Lemma~\ref{lem:CRP} and $(y,y^\prime)\in \widetilde{\cD}^{2\gamma}_{X,\alpha_1}$. Then the rough convolution  
            \begin{align}\label{def:RoughConv}
                \cI_t:=\int_0^t S_{t-r}Ny_r~\txtd \X_r:=\lim\limits_{|\mathcal{P}|\to 0} \sum\limits_{[u,v]\in\mathcal{P}} S_{t-u}N\left(y_{u}X_{v,u} + y'_{u}\mathbb{X}_{v,u}\right),
            \end{align}
            exists as an element of $\cB_{\alpha_2-2\gamma}$ for every $t\in [0,T]$, where $\mathcal{P}$ denotes a partition of $[s,t]$. Furthermore, we have for $0\leq s < t \leq T$ and $0\leq \beta<3\gamma$ the estimate
            \begin{align}\label{estimate:NewIntegral}
                \abs{\cR_{t,s}^{Ny}}_{\alpha_2-2\gamma+\beta} \lesssim \rho_\gamma(\X) \norm{Ny,Ny'}_{X,2\gamma,\alpha_2} (t-s)^{3\gamma-\beta},
            \end{align}
            where $\cR^{Ny}_{t,s}:=\int_{s}^{t} S_{t-r} Ny_{r}~\txtd\textbf{X}_{r} - S_{t-s}N\left(y_{s}X_{t,s} -y'_{s}\mathbb{X}_{t,s}\right)$ is the integral remainder. Consequently, $(\cI,Ny)\in \cD^{2\gamma}_{X,\alpha_2+ \theta}$ holds for every $\theta \in [0,\gamma)$.
        \end{corollary}
        \begin{proof}
            The existence of the rough convolution follows directly from Theorem~\ref{integral} and Lemma~\ref{lem:CRP}. Moreover, we obtain that $(\cI,Ny)\in \cD^{2\gamma}_{X,\alpha_2+ \theta}$ for every $\theta \in [0,\gamma)$ due to Corollary \ref{cor:higherreg}.
            \qed
        \end{proof}\\
        
        In order to make sense of~\eqref{eq:RoughConv} in the space of controlled rough paths, we need to make sure that $\cI_t\in D(A)$ for every $t\in[0,T]$ and to find a suitable Gubinelli derivative for $A \cI$. Therefore, we first specify the scales of Banach spaces which are required in our framework.\\ 
       We let from now on $p\in [2,3]$ and $2>\alpha>1+\frac{1}{p}$. In this case we define the spaces $\widetilde\cB_\beta:=B^{\beta-1-\frac{1}{p}}_{p,p}(\partial \cO)$ for $\beta \in \R$ and let $(A_\beta,\cB_\beta)_{\beta \in [-2,\infty)}$ be the interpolation-extrapolation scale generated by $A:D(A)\subset L^p(\cO)\to L^p(\cO)$. When $A$ satisfies Assumption \ref{ass}, we know that the scale of Banach spaces $\cB_\beta=D(A^\beta)$ can be expressed by
        \begin{align}\label{eq:BanachScale}
            \cB_{\frac{\beta}{2}}= H^{\beta,p}_\cC(\cO) :=
            \begin{cases}
                \{ u\in H^{\beta,p}(\cO) : \cC u =0 \}, & \beta>1+\frac{1}{p}\\
                H^{\beta,p}(\cO), & -1+\frac{1}{p}<\beta<1+\frac{1}{p}
            \end{cases},
        \end{align}
        for $-1+\frac{1}{p}<\beta\leq 2$, see for example \cite[Theorem 7.1]{Amann}. Recalling now that the Neumann operator $N$ is the solution operator of \eqref{eq:bvp}, we obtain for all $0<\varepsilon<\frac{1}{2}+\frac{1}{2p}$ that $N\in \cL(\widetilde\cB_{\alpha},\cB_\varepsilon)$, since $H^{\alpha,p}(\cO) \hookrightarrow  H^{2\varepsilon,p}(\cO) $. If we further replace $\alpha$ by $\alpha- \vartheta$ with $\vartheta\in \{\gamma, 2\gamma \}$, then $N$ is still bounded into $H^{\alpha-\vartheta,p}(\cO)=\cB_{\frac{\alpha-\vartheta}{2}}\hookrightarrow \cB_{\varepsilon-\vartheta}$, but $Ng$ is only a weak solution of the boundary value problem, see \cite[Section 9]{Amann}. This leads to the fact that $N\in \bigcap_{i=0}^2 \cL(\widetilde{\cB}_{\alpha-i\gamma},\cB_{\varepsilon-i\gamma})$, which means that Corollary~\ref{cor:RoughConv} is applicable with $\alpha_1:=\alpha$ and $\alpha_2:=\varepsilon$. 
        \begin{corollary}\label{cor:RoughConvInDomain}
            Let $0<\varepsilon<\frac{1}{2}+\frac{1}{2p}$ be arbitrary and $(y,y^\prime)\in \widetilde{\cD}^{2\gamma}_{X,\alpha}$. Then the rough integral $\cI_t$ defined by \eqref{def:RoughConv} belongs to $D(A)$ for every $t\in[0,T]$.
        \end{corollary} 
        \begin{proof} 
            According to Lemma \ref{lem:CRP} and Corollary \ref{cor:RoughConv} we get for every $\theta\in [0,\gamma)$ that $(\cI,Ny)\in \cD^{2\gamma}_{X,\varepsilon+\theta}$. We take $\theta:=\frac{1}{3}+\delta<\gamma$ with $\delta>0$ small enough. Then we can choose $\varepsilon:=\frac{1}{2}+\frac{1}{2p}-\delta$, so that we obtain $\varepsilon+\theta=\frac{5}{6}+\frac{1}{2p}\geq 1$ since $p\leq 3$. Therefore, we conclude that $\cI_t\in D(A)$ for every $t\in[0,T]$.
            \qed
        \end{proof}
        \begin{remark}\label{rem:dirichlet}
            \begin{itemize}
               \item [1)] The essential step in the proof of Corollary \ref{cor:RoughConvInDomain} is that we can choose an $\varepsilon$ such that $\varepsilon>1-\gamma$. Here we recall that $\gamma$ stands for the regularity of the noise. In the case of Neumann conditions, we have seen that this is possible. But for Dirichlet boundary conditions, we can choose $\varepsilon$ only up to $\frac{1}{2p}$. This comes from the fact that the Dirichlet-operator $\mathfrak{D}$ is  bounded from $B^{\beta-\frac{1}{p}}_{p,p}(\partial \cO)$ to $H^{\beta,p}(\cO)$ and provides a strong solution for $\beta>\frac{1}{p}$. Since $p\geq 2$ and $\gamma < \frac{1}{2}$, it is not possible to find an $\varepsilon$ such that $\varepsilon>1-\gamma$.
               \item   [2)] In the Young regime, i.e.~$\widetilde{\gamma}\in(\frac{1}2,1)$, we can incorporate Dirichlet boundary noise since the conditions $\varepsilon>1-\widetilde{\gamma}$ and $\varepsilon<\frac{1}{2p}$ can simultaneously be fulfilled. For additive fractional noise, it is known that Dirichlet boundary conditions can be incorporated provided that $H\in(\frac{3}{4},1)$ as established in~\cite{DuncanMaslowski}. We provide further details on the well-posedness of~\eqref{eq:intro} with multiplicative Dirichlet boundary noise in Theorem~\ref{globalDirichlet} and an example in Section \ref{examples}.
            \end{itemize}
        \end{remark}
        %Now we have to find a  suitable Gubinelli derivative for $A\cI$. 
        From now on we assume that $\varepsilon>1-\gamma$ so that $\cI_t\in D(A)$, as proved in Corollary~\ref{cor:RoughConvInDomain}, and set $\eta:=1-\varepsilon$, where $\varepsilon=\frac{1}{2} +\frac{1}{2p}-\delta$ for a small $\delta>0$. 
        \begin{remark}
        Since $Ny$ is a Gubinelli derivative for $\cI$, it would make sense to consider $ANy$ as one for $A\cI$. However, regarding the definition of $N$, $Ny$ does not belong to $D(A)$. Due to this reason, we need an extension of $A$, which is given by the extrapolated operator introduced in Section \ref{preliminaries}. In fact, $A_{-\eta}$ is the weakest possible extrapolation operator such that $A_{-\eta}Ny$ is well-defined.
        \end{remark}
        
        \begin{theorem}\label{thm:RoughConv}
            For every $(y,y^\prime)\in \widetilde{\cD}^{2\gamma}_{X,\alpha}$ we have $(A\cI,A_{-\eta}Ny)\in \cD^{2\gamma}_{X,-\eta}$.
        \end{theorem}
        \begin{proof}
            We set $ z_t:=A\cI_t$ and $z_t^\prime:=A_{-\eta}Ny_t$ for $t\in[0,T]$. By construction is $z_t\in \cB_0\hookrightarrow \cB_{-\eta}$ and $Ny_t\in \cB_\varepsilon= \cB_{1-\eta}$. Therefore, $z_t^\prime$ is well-defined with values in $\cB_{-\eta}\hookrightarrow\cB_{-\eta-\gamma}$. Consequently, we get $(z,z^\prime)\in C(\cB_{-\eta})\times C(\cB_{-\eta-\gamma})$. \\
            We recall that $A_{-\eta}$ can be viewed as the $\cB_{-\eta}$-realization of $A_{-\eta-2\gamma}$, see Remark~\ref{rem:Extrapolation}.~Then for $x\in \cB_{1-\eta}$, we have the equality $A_{-\eta-2\gamma}x=A_{-\eta}x$. Now we let $0\leq s<t\leq T$ and obtain, using additionally the fact that $A_{-\eta-2\gamma}\in \cL(\cB_{1-\eta-2\gamma},\cB_{-\eta-2\gamma})$
            \begin{align*}
                \abs{z^\prime_{t,s}}_{-\eta-2\gamma} =\abs{A_{-\eta}Ny_{t,s}}_{-\eta-2\gamma}\lesssim \abs{Ny_{t,s}}_{1-\eta-2\gamma}=\abs{Ny_{t,s}}_{\varepsilon-2\gamma}\lesssim \abs{Ny_{t,s}}_{\varepsilon-\gamma}\leq  \left[Ny\right]_{\gamma,\varepsilon-\gamma}\abs{t-s}^\gamma,
            \end{align*}
            where we used Remark \ref{rem:CRPy} in the last inequality. This means that $z^\prime \in C^\gamma( \cB_{-\eta-2\gamma})$. The tricky part is to estimate the remainder $R^z_{t,s}:= z_{t,s}-z^\prime_s X_{t,s}.$
            For this reason, we rewrite $R^z$, so that we can use the estimate \eqref{estimate:NewIntegral} derived for the integral remainder $\cR^{Ny}$:
            \begin{align}
                \begin{split}\label{eq:remainder}
                    R^z_{t,s}&=A\int_0^t S_{t-r}Ny_r~\txtd \X_r - A\int_0^s S_{s-r}Ny_r~\txtd \X_r-A_{-\eta}Ny_sX_{t,s}\\
                    & = A\left(\int_s^t S_{t-r}Ny_r~\txtd \X_r + (S_{t-s}-\Id )\int_0^s S_{s-r}Ny_r~\txtd \X_r  \right)+AS_{t-s}Ny^\prime_s\XX_{t,s}\\
                    & - AS_{t-s}Ny^\prime_s\XX_{t,s}+AS_{t-s}Ny_sX_{t,s}-AS_{t-s}Ny_sX_{t,s}-A_{-\eta}Ny_sX_{t,s}\\
                    &=A\cR^{Ny}_{t,s}+AS_{t-s}Ny^\prime_s \XX_{t,s}+(AS_{t-s}-A_{-\eta})Ny_sX_{t,s}+A(S_{t-s}-\Id)\cI_s.
                \end{split}
            \end{align}
            With this representation we get for $\vartheta\in \{\gamma, 2\gamma  \}$
            \begin{align*}
                \abs{R^z_{t,s}}_{-\eta-\vartheta}&\leq \underbrace{\abs{A\cR^{Ny}_{t,s}}_{-\eta-\vartheta}}_{I_1}+\underbrace{\abs{AS_{t-s}Ny^\prime_s \XX_{t,s}}_{-\eta-\vartheta}}_{I_2}+\underbrace{\abs{(AS_{t-s}-A_{-\eta})Ny_sX_{t,s}}_{-\eta-\vartheta}}_{I_3}\\
                &+\underbrace{\abs{A(S_{t-s}-\Id)\cI_s}_{-\eta-\vartheta}}_{I_4},
            \end{align*}
            so we can estimate the individual terms separately. Applying Corollary \ref{estimate:NewIntegral} with $\beta:=2\gamma-\vartheta$ entails
            \begin{align*}
                I_1&=\abs{A_{-\eta-\vartheta}\mathcal{R}^{Ny}_{t,s}}_{-\eta-\vartheta}\lesssim \abs{\cR^{Ny}_{t,s}}_{1-\eta -\vartheta}=\abs{\cR^{Ny}_{t,s}}_{\varepsilon-\vartheta}\lesssim \rho_\gamma(\X) \norm{Ny,Ny^\prime}_{X,2\gamma,\varepsilon}\abs{t-s}^{\gamma+ \vartheta}\\ 
                &\leq \rho_\gamma(\X) \norm{Ny,Ny^\prime}_{X,2\gamma,\varepsilon}\abs{t-s}^{\vartheta}T^\gamma,
            \end{align*}
             where we note that due to Corollary \ref{cor:RoughConvInDomain} it holds that $\cR^{Ny}_{t,s}\in D(A)=\cB_1$. To deal with the second term, we use $S_{t-s}\in\cL(\cB_{\varepsilon-\gamma})$ and $S_{t-s}Ny_s^\prime \in \cB_{\varepsilon-\gamma+1}\hookrightarrow \cB_1$ to further infer that
            \begin{align*}
                I_2&\lesssim \rho_\gamma(\X) \abs{S_{t-s}Ny_s^\prime}_{1-\eta-\vartheta}\abs{t-s}^{2\gamma}\lesssim \rho_\gamma(\X)  \abs{Ny_s^\prime}_{\varepsilon-\vartheta}\abs{t-s}^{2\gamma}\\ 
                &\lesssim \rho_\gamma(\X) \abs{Ny_s^\prime}_{\varepsilon-\gamma}\abs{t-s}^{2\gamma}\leq \rho_\gamma(\X) \norm{Ny,Ny^\prime}_{X,2\gamma,\varepsilon}\abs{t-s}^{\vartheta}T^{2\gamma-\vartheta}.
            \end{align*}
            Combining Remark \ref{rem:Extrapolation} with the smoothing property \eqref{hg:1} to get the estimate
            \begin{align*}
                I_3&\lesssim \rho_\gamma (\X)  \abs{(AS_{t-s}-A_{-\eta})Ny_s}_{-\eta-\vartheta}\abs{t-s}^\gamma\\
                &=\rho_\gamma (\X)  \abs{A_{-\eta-\vartheta}(S_{t-s}-\Id)Ny_s}_{-\eta-\vartheta}\abs{t-s}^\gamma\\
                &\lesssim \rho_\gamma (\X) \abs{(S_{t-s}-\Id)Ny_s}_{1-\eta-\vartheta}\abs{t-s}^\gamma \\
                &= \rho_\gamma (\X) \abs{(S_{t-s}-\Id)Ny_s}_{\varepsilon-\vartheta}\abs{t-s}^\gamma \\
                & \lesssim \rho_\gamma (\X) \|S_{t-s}-\Id\|_{\cL(\cB_\varepsilon,\cB_{\varepsilon-\vartheta})} |Ny_s|_{\varepsilon} |t-s|^\gamma \\
                &\lesssim \rho_\gamma(\X) \|Ny\|_{\infty,\varepsilon}\abs{t-s}^{\gamma+\vartheta}\\
                &\leq \rho_\gamma(\X) \norm{Ny,Ny^\prime}_{X,2\gamma,\varepsilon} \abs{t-s}^{\vartheta}T^{\gamma}.
            \end{align*}
            In order to estimate $I_4$ we first apply~\eqref{hg:1} to obtain
            \begin{align*}
                I_4&=\abs{A_{-\eta-\vartheta}(S_{t-s}-\Id)\cI_s}_{-\eta-\vartheta}\lesssim\abs{(S_{t-s}-\Id)\cI_s}_{1-\eta-\vartheta}\\
                &= \abs{(S_{t-s}-\Id)\cI_s}_{\varepsilon-\vartheta}\lesssim \abs{\cI_s}_{\varepsilon} \abs{t-s}^\vartheta,
            \end{align*}
            \tim{where we used again Remark \ref{rem:Extrapolation} for $A(S_{t-s}-\Id)\cI_s=A_{-\eta-\vartheta}(S_{t-s}-\Id)\cI_s$}. Now we can use a similar decomposition as in \eqref{eq:remainder} for $\cI_s$. This leads, together with \eqref{hg:2}, \eqref{estimate:NewIntegral} and the fact that $S_t\in \cL(\cB_\varepsilon)$ to
            \begin{align*}
                I_4&\lesssim \left(\abs{\cR^{Ny}_{s,0}}_{\varepsilon}+\abs{S_sNy_0X_{s,0}}_{\varepsilon}+\abs{S_sNy_0^\prime \XX_{s,0}}_{\varepsilon} \right)\abs{t-s}^{\vartheta}\\
                &\lesssim \rho_\gamma(\X)\left(\norm{Ny,Ny^\prime}_{X,2\gamma,\varepsilon}s^\gamma+\abs{S_sNy_0}_\varepsilon s^\gamma+\abs{S_sNy_0^\prime}_\varepsilon s^{2\gamma}\right)\abs{t-s}^{\vartheta}\\
                &\lesssim \rho_\gamma(\X)\left(\norm{Ny,Ny^\prime}_{X,2\gamma,\varepsilon}s^\gamma+\norm{Ny,Ny^\prime}_{X,2\gamma,\varepsilon} s^\gamma+\abs{Ny_0^\prime}_{\varepsilon-\gamma} s^{\gamma}\right)\abs{t-s}^{\vartheta}\\ 
                &\leq 3 \rho_\gamma(\X) \norm{Ny,Ny^\prime}_{X,2\gamma,\varepsilon} \abs{t-s}^{\vartheta} T^\gamma.
            \end{align*}
            Putting all the previous estimates together, we conclude that
            \begin{align*}
                \left[R^z\right]_{\vartheta,-\eta-\vartheta}\lesssim \rho_\gamma (\X) \norm{Ny,Ny^\prime}_{X,2\gamma,\varepsilon},
            \end{align*}
            so $R^z\in C^\gamma(\cB_{-\eta-\gamma})\cap C^{2\gamma}(\cB_{-\eta-2\gamma})$ which completes the proof.
            \qed\\
        \end{proof}\\
        %Similar to~\cite[Corollary 4.6]{GHN} one can show the continuity of the rough integration as a map from $\widetilde{\cD}^{2\gamma}_{X,\alpha}$ to $\cD^{2\gamma}_{X,-\eta}$.~Here, the rough convolution~\eqref{def:RoughConv} does not increase the spatial regularity, since we need that $\cI$ takes values in $D(A)$.\\
        Even if Theorem~\ref{thm:RoughConv} is an interesting result on its own, the statement is not enough for our purposes due to the presence of the operator $A$ in front of the rough integral. In particular, it is not possible to show that~\eqref{eq:intro} has a global solution working with the controlled rough path $(A\cI, A_{-\eta}Ny)\in \cD^{2\gamma}_{X,-\eta}$ and using the techniques in~\cite{HN21}, even though we could establish a local solution using a fixed-point argument.
        Therefore, we further show that we can plug the operator $A$ in the rough integral, see~\cite{m} for an analogous result for additive fractional noise. As a consequence of Corollary \ref{cor:RoughConv}, the limit on the right-hand side of \eqref{def:RoughConv}, exists in $\cB_{\varepsilon-2\gamma}$. So
        the equality 
        \begin{align}\label{interchangeIntegral}
            \widetilde{A}\int_0^{t} S_{t-r} N(y_r)~\txtd\X_r=\int_0^{t} \widetilde{A}S_{t-r} N(y_r)~\txtd\X_r,
        \end{align}
        holds for a bounded, and therefore continuous, operator $\widetilde{A}$ with domain $\cB_{\varepsilon-2\gamma}$.
        However, in our case, we only have $A\in \cL(\cB_1,\cB_0)$ and $\varepsilon-2\gamma<1$. Nevertheless, we can show the following statement. 
        \begin{lemma}\label{lemma:3.7}
            Under the assumptions of Corollary \ref{cor:RoughConv}, the limit 
            \begin{align*}
                \int_s^t S_{t-r}Ny_r~\txtd \X_r= \lim\limits_{|\mathcal{P}|\to 0} \sum\limits_{[u,v]\in\mathcal{P}} S_{t-u}N\left(y_{u}X_{v,u} + y'_{u}\mathbb{X}_{v,u}\right),
            \end{align*}
            exists $0\leq s<t\leq T$ in the $\cB_{\alpha_2-2\gamma+\beta}$ topology for every $\beta\in [0,3\gamma)$.
        \end{lemma}
        \begin{proof}
            For the sake of completeness, we indicate a sketch of the proof of this statement based on a classical sewing lemma, see~\cite[Theorem 4.1]{GHN} and~\cite[Theorem 2.4]{GHairer}. Let $\mathcal{P}^n:=\{t_i=s+2^{-n} i (t-s)~:~i=0,\ldots, 2^n \}$ be the $n$-th dyadic partition of the interval $[s,t]$, and $\cI^{\mathcal{P}^n}_{t,s}:=\sum_{[u,v]\in\mathcal{P}^n} S_{t-u}N\left(y_{u}X_{v,u} + y'_{u}\mathbb{X}_{v,u}\right):=\sum_{[u,v]\in\mathcal{P}^n} S_{t-u}\xi_{v,u}$ the sum associated to the partition $\mathcal{P}^n$. To prove now that
            $(\cI_{t,s}^{\mathcal{P}^n})_{n\in \N}$ is a Cauchy sequence, set $m=\frac{v-u}{2^{n+1}}$ the midpoint of an interval $[u,v]$. Then we can write
            \begin{align*}
                \cI^{\mathcal{P}^n}_{t,s}-\cI^{\mathcal{P}^{n+1}}_{t,s}&=\sum_{[u,v]\in\mathcal{P}^n} S_{t-u}\xi_{v,u}-\sum_{[u,v]\in\mathcal{P}^n} S_{t-u}\xi_{m,u}+S_{t-m}\xi_{v,m}\\
                &=\sum_{[u,v]\in\mathcal{P}^n} S_{t-u}(\xi_{v,u}-\xi_{m,u}-\xi_{v,m})+S_{t-m}(S_{m-u}-\textrm{id})\xi_{v,m}.
            \end{align*}
            With this representation one can show that
            \begin{align}\label{est:clear}
                \abs{\cI^{\mathcal{P}^n}_{t,s}-\cI^{\mathcal{P}^{n+1}}_{t,s}}_{\alpha_2-2\gamma+\beta}\leq C_\xi 2^{-n(3\gamma-1-\delta)} \abs{t-s}^{3\gamma-\beta},
            \end{align}
            for a $\delta\in (\beta-1,3\gamma-1)$, using similar ideas as in \cite[Theorem 4.1, 4.5]{GHN}. The only difference is the appearance of the operator $N\in \bigcap_{i=0}^2 \cL(\widetilde{\cB}_{\alpha_1-i\gamma},\cB_{\alpha_2-i\gamma})$, which is bounded and therefore only changes the space we end up with. For instance, we consider the first part of the sum. With the help of Chen's relation it can be shown that
            \begin{align*}
                \xi_{v,u}-\xi_{m,u}-\xi_{v,m}=R^{Ny}_{u,m}X_{v,m}+Ny^\prime_{u,m}\XX_{v,m},
            \end{align*}
            and therefore, with the regularity property \eqref{hg:2} and the H\"older conditions of the controlled rough path $(Ny,Ny^\prime)\in\cD^{2\gamma}_{X,\alpha_2}$, one gets
            \begin{align*}
                &\abs{ \sum_{[u,v]\in\mathcal{P}^n} S_{t-u}(\xi_{v,u}-\xi_{m,u}-\xi_{v,m})}_{\alpha_2-2\gamma+\beta}\\ &\leq \sum_{[u,v]\in\mathcal{P}^n} \abs{S_{t-u}R^{Ny}_{u,m}X_{v,m}}_{\alpha_2-2\gamma+\beta}+\abs{S_{t-u}Ny^\prime_{u,m}\XX_{v,m}}_{\alpha_2-2\gamma+\beta}\\ 
                &\lesssim \rho_\gamma (\X) \norm{Ny,Ny^\prime}_{X,2\gamma,\alpha_2}\sum_{[u,v]\in\mathcal{P}^n}\abs{t-m}^{-\beta}(\abs{v-m}^\gamma\abs{m-u}^{2\gamma}+\abs{v-m}^{2\gamma}\abs{m-u}^\gamma).
            \end{align*}
            The second term can be treated analogously which means that~\eqref{est:clear} holds for a constant $C_\xi$ which depends on the H\"older norms of $\xi$ and on the semigroup. Furthermore, since the right-hand side of~\eqref{est:clear} is summable over $n$, the sequence $(\cI_{t,s}^{\mathcal{P}^n})_{n\in \N}$ is Cauchy in $\cB_{\alpha_2-2\gamma+\beta}$ and therefore has a limit $\widetilde{\cI}_{t,s}\in \cB_{\alpha_2-2\gamma+\beta}$. Since $\cB_{\alpha_2-2\gamma+\beta}\hookrightarrow \cB_{\alpha_2-2\gamma}$, and the limit in Corollary~\ref{cor:RoughConv} is unique, we get that $\widetilde{\cI}_{t,s}=\cI_{t,s}$. In conclusion, the limit exists in the $\cB_{\alpha_2-2\gamma+\beta}$ topology.
            \qed
        \end{proof}\\
        
        Now, going back to the situation in~\eqref{interchangeIntegral}, we choose $\beta:=2\gamma-\varepsilon+1<3\gamma$, due to our restriction on $\varepsilon$. Since $A$ satisfies Assumption \ref{ass}, Theorem \ref{thm:EpolSemigroup} ensures that every extrapolated operator generates again an analytic semigroup. Together with Theorem \ref{thm:EpolOperator}, Remark \ref{rem:Extrapolation} and the fact that $Ny\in \cB_{\varepsilon}=\cB_{1-\eta}$, this leads to
        \begin{align}\label{newInterchangeIntegral}
            A\int_0^{t} S_{t-r} Ny_r~\txtd\X_r=\int_0^{t} S_{t-r}A_{-\eta} Ny_r~\txtd\X_r.
        \end{align}
       
      \begin{remark}\label{two:i}
          To make sure that the right-hand side is well-defined as a controlled rough integral, we need to find a Gubinelli derivative for $A_{-\eta}Ny$. A natural choice would be $A_{-\eta}Ny^\prime$, but since $y^\prime$ loses spatial regularity, this is not well-defined. Therefore, in order to choose an appropriate Gubinelli derivative for $A_{-\eta}N y$, the extrapolated operator $A_{-\eta}$ has to be lifted. Due to this reason, one can show that $(A_{-\eta}Ny, A_{-\sigma}Ny')\in \cD^{2\gamma}_{X,-\eta}$ holds with $\sigma:=\eta+\gamma$. In order to avoid working with two different indices for the extrapolation operator in the path component and its Gubinelli derivative, we rely on Remark \ref{rem:Extrapolation}. Therefore we have $A_{-\eta}Ny=A_{-\sigma}Ny$ since $-\eta>-\sigma$ and $Ny\in \cB_\varepsilon=\cB_{1-\eta}\hookrightarrow\cB_{1-\sigma}$.
      \end{remark}
        This enables us to formulate the next result.
    
            %Based on our considerations, it is evident that $-\eta$ is the correct index in the path $A_{-\eta}Ny$, as is $-\sigma-\gamma$ as soon as the Gubinelli derivative $y^\prime$ appears. But this leads to the fact, that we have to pair $A_{-\eta}Ny$ with a Gubinelli derivative that is not usual, and we had to remake some classical results for such controlled rough paths. In order to avoid that, we again use the property in Remark      %   \end{remark}
        \begin{lemma}\label{lem:GubDeriv}
            For every $(y,y^\prime)\in \widetilde{\cD}^{2\gamma}_{X,\alpha}$ we have $(A_{-\sigma}Ny,A_{-\sigma}Ny^\prime)\in \cD^{2\gamma}_{X,-\eta}$.
        \end{lemma}
        \begin{proof}
            Note that we cannot use Lemma \ref{lem:CRP}, since $A_{-\sigma}$ is not defined for every element in $\cB_{\varepsilon-2\gamma}$. So we have to take advantage of the fact that $Ny_t\in \cB_{\varepsilon}=\cB_{1-\eta}$ and $Ny_t^\prime\in \cB_{\varepsilon-\gamma}=\cB_{1-\sigma}$. This leads to  $z_t:=A_{-\sigma}Ny_t\in\cB_{-\eta}$ and $z^\prime_t:=A_{-\sigma}Ny^\prime_t\in \cB_{-\sigma}=\cB_{-\eta-\gamma}$. Furthermore, we have due to Remark \ref{rem:Extrapolation} 
            \begin{align*}
                \abs{z'_{t,s}}_{-\eta-2\gamma}=
            \abs{A_{-\sigma} N y^\prime_{t,s}}_{-\eta-2\gamma}&=\abs{A_{-\eta-2\gamma} N y^\prime_{t,s}}_{-\eta-2\gamma}\lesssim \abs{N y^\prime_{t,s}}_{1-\eta-2\gamma}\\
            &=\abs{N y^\prime_{t,s}}_{\varepsilon-2\gamma}\lesssim \abs{t-s}^{\gamma} \left[Ny^\prime\right]_{\gamma,\varepsilon-2\gamma}.
            \end{align*}
            To investigate the remainder $R^z_{t,s}:=A_{-\sigma}R^{Ny}_{t,s}\in \cB_{1-\sigma}$, we let $\vartheta\in \{\gamma,2\gamma\}$ and establish 
            \begin{align*}
                \abs{R^z_{t,s}}_{-\eta-\vartheta}&=\abs{A_{-\sigma}R^{Ny}_{t,s}}_{-\sigma -(\vartheta-\gamma)}=\abs{A_{-\sigma-(\vartheta-\gamma)}R^{Ny}_{t,s}}_{-\sigma -(\vartheta-\gamma)}\lesssim \abs{R^{Ny}_{t,s}}_{1-\sigma-(\vartheta-\gamma)}\\ 
                &=\abs{R^{Ny}_{t,s}}_{\varepsilon-\vartheta}\lesssim \abs{t-s}^\vartheta \left[R^{Ny}\right]_{\vartheta, \varepsilon -\vartheta},
            \end{align*}
            using again Remark \ref{rem:Extrapolation}. Regarding the previous deliberations, this computation concludes the proof.
            \qed
        \end{proof}\\
        
      Consequently this allows us to define the rough convolution.  
       \begin{lemma}\label{sigma}
           The right-hand side of~\eqref{newInterchangeIntegral} is well-defined as a rough convolution with the controlled rough path $(A_{-\sigma} Ny, A_{-\sigma} N y')\in \cD^{2\gamma}_{X,-\eta}$. Moreover, a Gubinelli derivative of
            $\int_0^t S_{t-r} A_{-\sigma} Ny_r~\txtd \mathbf{X}_r$ is $A_{-\sigma} N y_t$.
       \end{lemma}
        
        Further, since we want to solve equations with multiplicative noise, the next step is to consider the composition of a controlled rough path with a smooth function. In our setting, in contrast to~\cite{GHN}, the nonlinearity is allowed to map between different scales of Banach spaces.
        \begin{lemma} (Composition of a controlled rough path with a smooth function)\label{lem:composition}
            Let $\beta,\delta\in \R$ and $F:\cB_{\beta-\vartheta}\to \widetilde{\cB}_{\beta-\vartheta+\delta}$ two times continuously Fr\'{e}chet differentiable with bounded derivatives for any $\vartheta \in \{0,\gamma, 2\gamma \}$. For $(y,y^\prime)\in \cD^{2\gamma}_{X,\beta}$ we define $(z_t,z_t^\prime):=(F(y_t),DF(y_t)\circ y_t^\prime)$ for $t\in [0,T]$.
            \begin{itemize}
                \itemsep -2pt
                \item[i)] We have $(z,z^\prime)\in \widetilde{\cD}^{2\gamma}_{X,\beta+\delta}$ and the estimate
                \begin{align}\label{est:SmoothComposition}
                     \norm{z,z^\prime}_{X,2\gamma,\widetilde{\cB}_{\beta+\delta}}\lesssim \norm{F}_{C^2}(1+\rho_\gamma(\X))^2 \norm{y,y^\prime}_{X,2\gamma,\cB_\beta}(1+\norm{y,y^\prime}_{X,2\gamma,\cB_\beta}),
                \end{align}
                holds.
                \item[ii)] Assume additionally that $F$ is three times Fr\'{e}chet differentiable with bounded third derivative, and let $(\widetilde{z},\widetilde{z}^\prime)$ be the composition of another controlled rough path $(\widetilde{y},\widetilde{y}^\prime)\in \cD^{2\gamma}_{X,\beta}$ with $F$. Then
                \begin{align}\label{est:SmoothComposition2}
                    \norm{z-\widetilde{z},z^\prime-\widetilde{z}^\prime}_{X,2\gamma,\widetilde{\cB}_{\beta+\delta}} \lesssim \norm{F}_{C^3}&(1+\rho_\gamma(\X))^2(1+\norm{y,y^\prime}_{X,2\gamma,\cB_\beta}+\norm{\widetilde{y},\widetilde{y}^\prime}_{X,2\gamma,\cB_\beta})^2\nonumber\\ &\times \norm{y-\widetilde{y},y^\prime-\widetilde{y}^\prime}_{X,2\gamma,\cB_\beta}
                \end{align}
                is satisfied.
            \end{itemize}
            Where we set $\norm{F}_{C^k}:=\max\limits_{\vartheta\in \{\gamma,2\gamma\}}  \norm{F}_{C^k(\cB_{\beta-\vartheta},\widetilde{\cB}_{\beta-\vartheta+\delta})}$.
        \end{lemma}
        \begin{proof}
            The proof is similar to \cite[Lemma 4.7]{GHN}. We point out the main differences that occur in our case. We can view for $t\in [0,T]$ the derivative $D^kF(y_t)$ as an element of $\cL(\cB^{\otimes k}_{\beta-\vartheta},\widetilde{\cB}_{\beta-\vartheta+\delta})$ for $k=1,2,3$ and $\vartheta\in \{\gamma,2\gamma\}$. Since $(\cB_\beta)_{\beta\in\R}$ and $(\widetilde{\cB}_\beta)_{\beta\in\R}$ are both Banach scales (recall Definition~\ref{raum}), all the necessary estimates remain valid. For instance, one can estimate the Gubinelli derivative $z'=DF(y)\circ y'$ as
            \begin{align*}
                \abs{z^\prime_{t,s}}_{\widetilde{\cB}_{\beta-2\gamma+\delta}}&\leq \abs{DF(y_s)\circ y_{t,s}^\prime}_{\widetilde{\cB}_{\beta-2\gamma+\delta}} + \abs{(DF(y_t)-DF(y_s))\circ y_t^\prime}_{\widetilde{\cB}_{\beta-2\gamma+\delta}}\\ 
                &\leq \norm{DF(y_s)}_{\cL(\cB_{\beta-2\gamma},\widetilde{\cB}_{\beta-2\gamma+\delta})}\abs{y_{t,s}^\prime}_{\cB_{\beta-2\gamma}}\\
                &+\norm{DF(y_t)-DF(y_s)}_{\cL(\cB_{\beta-2\gamma},\widetilde{\cB}_{\beta-2\gamma+\delta})}\abs{y_s^\prime}_{\cB_{\beta-2\gamma}}\\ 
                &\leq \norm{F}_{C^1} \left[y^\prime\right]_{\gamma,\cB_{\beta-2\gamma}}\abs{t-s}^\gamma+ \norm{D^2F}_{C} \abs{y_{t,s}}_{\cB_{\beta-2\gamma}}\norm{y^\prime}_{\infty,\cB_{\beta-\gamma}} \\ 
                &\lesssim \norm{F}_{C^1} \norm{y,y^\prime}_{X,2\gamma,\cB_\beta} \abs{t-s}^\gamma + \norm{F}_{C^2} \norm{y,y^\prime}_{X,2\gamma,\cB_\beta} \left[y \right]_{\gamma,\cB_{\beta-2\gamma}} \abs{t-s}^\gamma\\ 
                &\leq \norm{F}_{C^1} \norm{y,y^\prime}_{X,2\gamma,\cB_\beta} \abs{t-s}^\gamma \\ 
                &+ \norm{F}_{C^2} \norm{y,y^\prime}_{X,2\gamma,\cB_\beta} (1+\rho_\gamma(\X))\norm{y,y^\prime}_{X,2\gamma,\cB_\beta} \abs{t-s}^\gamma,
            \end{align*}
            where we use \eqref{est:hoelder:y}. The estimates of the remainder follow by analogue computations. 
            \qed\\
        \end{proof}
         
        Returning to~\eqref{eq:intro} and regarding that according to~\eqref{newInterchangeIntegral} and Lemma~\ref{sigma} it holds  $$A\int_0^{t} S_{t-r} NF(y_r)~\txtd\X_r=\int_0^{t} S_{t-r}A_{-\sigma} NF(y_r)~\txtd\X_r,$$
        and we can now rewrite \eqref{eq:intro} as a semilinear evolution equation without boundary noise %\textcolor{blue}{explain why $A_{-\eta}$}
        \begin{align}\label{eq:introNew}
            \begin{cases}
                  \txtd y = \left(Ay + f(y)\right)~\txtd t + A_{-\sigma}NF(y)~\txtd \X_t,\\
                  y(0)=y_0\in \cB_{-\eta}.
            \end{cases}
        \end{align}
        We recall that $A_{-\sigma}$ is the extrapolation operator introduced in Section~\ref{preliminaries}, $\sigma=\eta+\gamma=1-\varepsilon+\gamma$ with $\varepsilon=\frac{1}{2}+\frac{1}{2p}-\delta$ for a small $\delta>0$.% and $\eta=1-\varepsilon$.         %\textcolor{blue}{Comment $A_{-\eta}$ and $A_{-\sigma}$}.

        \begin{remark}
            The idea to rewrite~\eqref{eq:intro} as a semilinear problem without boundary noise as in~\eqref{eq:introNew} using an extrapolation operator was also applied in~\cite{VeraarSchnaubelt}. There the extrapolated operator $A_{-1}$ was used. We note that the index of the extrapolation operator needed there is $\frac{\widetilde{\alpha}}{2}-1$ where $\widetilde{\alpha} \in \left(1,1+\frac{1}{p}\right)$. Therefore, our result is consistent with the one in \cite{VeraarSchnaubelt} for Brownian noise, since in both cases the index satisfies $-\eta=\frac{\widetilde{\alpha}}{2}-1\in \left(-\frac{1}{2},-\frac{1}{2}+\frac{1}{2p}\right)$ due to the restriction $\varepsilon>1-\gamma$ and $\gamma\in \left(\frac{1}{3},\frac{1}{2}\right]$. We work here with the extrapolation operator $A_{-\eta-\gamma}$, as pointed out in Remark~\ref{two:i} because this seems to fit well in the rough path framework. 
        \end{remark}
       % Consequently, we are now able to use the theory developed in \cite{GHN} for local and \cite{HN19,HN20} for global solutions.
       
        We give now the main assumptions on the nonlinear drift and diffusion coefficients of \eqref{eq:intro}, that will guarantee the local-~as well as the global-in-time existence of solutions based on the results in~\cite{GHN,HN21}.
        \begin{assumptions}\label{ass2}
            \begin{itemize}
                \item[1)] There exists a $\delta_1\in[0,1)$ such that the drift term $f:\cB_{-\eta}\to \cB_{-\eta-\delta_1}$ is Lipschitz continuous.
                \item[2)] There exists a $\delta_2>\eta+1+\frac{1}{p}$ such that for any $\vartheta \in \{0,\gamma, 2\gamma \}$ the diffusion term $F:\cB_{-\eta-\vartheta}\to \widetilde{\cB}_{-\eta-\vartheta+\delta_2}$ is three times continuously  Fr\'{e}chet differentiable with bounded derivatives.
                % \textcolor{blue}{for completeness we should add somewhere a remark comparing this assumption to the assumptions in~\cite{VeraarSchnaubelt} for Brownian motion. Our is more restrictive to the rough path approach.}
                \item[3)] Let 1) and 2) be satisfied. Assume additionally that $f:\cB_{-\eta}\to \cB_{-\eta-\delta_1}$ for $\delta_1\in[2\gamma,1)$ satisfies a linear growth condition and that the derivative of 
                \begin{align*}
                    DF(\cdot) \circ G(\cdot):\cB_{-\eta-\gamma}\to \widetilde{\cB}_{-\eta-\gamma+\delta_2}
                \end{align*}
                %\textcolor{blue}{if we assume $\cB_{-\eta-\gamma}\to\cB_{-\eta-\gamma+\delta_2}$ should be the correct assumption, under that i can prove a similar result as in \cite{HN21}} 
                is bounded, where $G:=A_{-\sigma} N F$. %\textcolor{blue}{write here where $G$ maps and that is well-defined, we have this above in the proof of Thm 3.14, we could move it here}
            \end{itemize}
        \end{assumptions}
        \begin{remark}
            \begin{itemize}
                \item[i)] The assumption on the diffusion coefficient ensures that $-\eta+\delta_2>1+\frac{1}{p}$ such that $NF:\cB_{-\eta}\to \cB_{\varepsilon}=\cB_{1-\eta}\hookrightarrow \cB_{1-\sigma}$ is bounded due to the definition of the Neumann operator. Therefore $G:\cB_{-\eta}\to\cB_{-\eta}$ is well-defined. Since $NF:\cB_{-\eta-\gamma}\to \cB_{\varepsilon-\gamma}=\cB_{1-\sigma}$ is also valid, we conclude  that $G:\cB_{-\eta-\gamma}\to\cB_{-\eta-\gamma}$ is also well-defined. However, it is not true that $G:\cB_{-\eta-2\gamma}\to \cB_{-\eta-2\gamma}$, since the extrapolated operator $A_{-\sigma}$ is no longer defined on this space. Nevertheless, we  can deal with this technical issue, recall Lemma~\ref{lem:GubDeriv} for a similar situation. 
             
                \item[i)] Due to the rough path techniques, the assumptions on the diffusion coefficient $F$ are more restrictive than in \cite{VeraarSchnaubelt}, where $F$ maps into $\cB_{-1}$, is Lipschitz continuous and satisfies a linear growth condition. Moreover, due to the presence of the Neumann operator, $F$ is supposed to improve the spatial regularity, in order to define the rough convolution as in~\eqref{Gintegral}. Such issues are common for rough convolutions and have been also encountered in~\cite{HN19,Maria}. 
                
                %This aspect goes beyond the purposes of this work. 
                %On the other hand, they are dealing with stochastic integration in Banach spaces against some Brownian motion. If this is just a one-dimensional, then $F$ has to map $L^p$ functions into $\cB_{-1}$, where the Banach scale $(A_\alpha, \cB_\alpha)_{\alpha\in [-2,\infty)}$ is, as in our case, generated by $(A,D(A))$. But they consider infinite-dimensional noise, so they need additional assumptions.
                %\textcolor{blue}{if the noise in [30] was just a one dimensional scalar Brownian motion, then F should map into $\cB_{-1}$. We can just state that and the fact that they consider infinite-dimensional noise.}
                % \textcolor{blue}{should we mention that due to the stochastic integration in Banach spaces, they also need some $\gamma$-radonifying condition? Yes, and also where $F$ has to map from and into.}
            \end{itemize}
        \end{remark}

        \begin{corollary}\label{new:crp}
            Let Assumption \ref{ass2} 2) be fulfilled. Then for every $(y,y^\prime)\in \cD^{2\gamma}_{X,-\eta}$ we have $(A_{-\sigma}NF(y),A_{-\sigma}N(DF(y)\circ y^\prime))\in \cD^{2\gamma}_{X,-\eta}$. 
        \end{corollary}
        \begin{proof}
            Let be $(y,y^\prime)\in \cD^{2\gamma}_{X,-\eta}$. Due to Lemma \ref{lem:composition} we know that $(F(y),DF(y)\circ y^\prime)\in \widetilde{\cD}^{2\gamma}_{X,-\eta+\delta_2}$. Since $-\eta+\delta_2>1+\frac{1}{p}$, the claim follows applying Lemma \ref{lem:GubDeriv} to the controlled rough path $(F(y), DF(y)\circ y')$.
            \qed
        \end{proof}
        \begin{theorem}\em{(Existence of a local-in-time solution for~\eqref{eq:introNew})}\label{local}
            Assume that $F$ and $f$ satisfy Assumption \ref{ass2} 1) and 2). Then there exists for every initial condition $y_0\in \cB_{-\eta}$ a time $T^*\leq T$ and a unique solution $(y,A_{-\sigma}N F(y))\in \cD^{2\gamma}_{X,-\eta}\left([0,T^*)\right)$ to \eqref{eq:intro} such that 
            \begin{align}\label{eq:mildSolution}
                y_t=S_t y_0 + \int_0^t S_{t-r} f(y_r)~\txtd r+\int_0^t S_{t-r} A_{-\sigma}NF(y_r)~\txtd \X_r, \text{ for all } t<T^*.
            \end{align}
           % where $A_{-\sigma}$ is the extrapolation operator introduced in Section~\ref{preliminaries}, $\sigma=1+\gamma-\varepsilon$, $\varepsilon=\frac{1}{2}+\frac{1}{2p}-\delta$ for a small $\delta>0$.
        \end{theorem}
        \begin{proof} 
        %Note that $(F(y),DF(y)\circ y^\prime)\in \widetilde{\cD}^{2\gamma}_{X,-\eta+\delta_2}$ by Lemma \ref{lem:composition} and that $(A_{-\sigma}NF(y), A_{-\sigma} N(DF(y)\circ y'))\in \cD^{2\gamma}_{X,-\eta} $ due to Corollary~\ref{new:crp}. Therefore one can apply a fixed-point argument in $\cD^{2\gamma}_{X,-\eta}$. 
        To prove this, we use the Banach fixed point theorem. Therefore, we show that the map 
        \begin{align*}
            \Phi(u,u^\prime):=\left(S_\cdot y_0 + \int_0^\cdot S_{\cdot-r} f(u_r)~\txtd r+\int_0^\cdot S_{\cdot-r} A_{-\sigma}NF(u_r)~\txtd \X_r,A_{-\sigma}NF(u_\cdot)\right)
        \end{align*}
        with $(u,u^\prime)\in \cD^{2\gamma}_{X,-\eta}$ admits for some $\tau\leq T$ a fixed point in the closed ball
        \begin{align*}
            B_\tau(y_0):=\{(u,u^\prime)\in\cD^{2\gamma}_{X,-\eta}\left([0,\tau]\right)~:~ (u_0,u_0^\prime)=(y_0,A_{-\sigma}NF(y_0)) \text{ and } \norm{u-\xi,u^\prime-\xi^\prime}_{X,2\gamma,-\eta}\leq 1 \}
        \end{align*}
        centered around $(\xi,\xi^\prime)$ with 
        $$\xi_t:=S_{t}y_0+\int_0^t S_{t-r}A_{-\sigma}NF(y_0)~\txtd \X_r \qquad \text{and} \qquad \xi_t^\prime:= A_{-\sigma}NF(y_0).$$
        One can directly see that $ (\xi,\xi^\prime)\in\cD^{2\gamma}_{X,-\eta}\left([0,T]\right)$.
    ~The strategy is to prove first that $\Phi$ leaves $B_\tau(y_0)$ invariant and that $\Phi$ is a contraction on $B_\tau(y_0)$ for a sufficiently small $\tau$. Then Banach's fixed point theorem ensures the existence of a unique fixed point $(y,y^\prime)\in\cD^{2\gamma}_{X,-\eta}\left([0,\tau]\right)$, where $y$ satisfies \eqref{eq:mildSolution} and $y^\prime=A_{-\sigma}NF(y_t)$. Both invariance and contraction property of $\Phi$ can be obtained using that $(F(y),DF(y)\circ y^\prime)\in \widetilde{\cD}^{2\gamma}_{X,-\eta+\delta_2}$ by Lemma~\ref{lem:composition} and therefore $(A_{-\sigma}NF(y), A_{-\sigma} N(DF(y)\circ y'))\in \cD^{2\gamma}_{X,-\eta} $ due to Corollary~\ref{new:crp}. Due to these results, we are able to use the estimates for the rough convolution in Lemma~\ref{lem:composition} and Corollary~\ref{cor:higherreg} similar to the semilinear case \cite[Theorem 5.1]{GHN} or~\cite[Theorem 4.1]{GHairer}. 
            \qed
    \end{proof}\\
        % We only have to show that 
            % \begin{align}\label{eq:Nonlin}
            %     G:=A_{-\eta}NF:\cB_{-\eta-\vartheta} \to\cB_{-\eta-\vartheta},    
            % \end{align}
            % \textcolor{blue}{we don't need this property, since (G(y),G(y)') is a controlled RP)} for $\vartheta \in \{0,\gamma, 2\gamma \}$ is three times continuously differentiable with bounded derivatives and compute a Gubinelli derivative for $G(y)$. 
            
            % Since $F$ satisfies Assumption \ref{ass} 2) and $A_{-\sigma}$ and $N$ are bounded and linear operators, the boundedness and differentiability assumptions follow immediately. Since $\delta_2>1+\frac{1}{p}+ \eta$, we have $\widetilde{\cB}_{-\eta+\delta_2}\hookrightarrow \widetilde{\cB}_{\alpha}$ for an $\alpha>1+\frac{1}{p}$. Due to the definition of the Neumann operator $N$, we know that $NF$ maps $\cB_{-\eta}$ to $\cB_{\varepsilon}=\cB_{1-\eta}\hookrightarrow \cB_{1-\sigma}$. Therefore, $G=A_{-\sigma}NF$ is well-defined. 
            %So due to Lemma \ref{lem:GubDeriv} we define the Gubinelli derivative for $G(y)$ by $G(y)':=A_{-\sigma} N (DF(y) \circ y')$.
              % \end{proof}\\
        
        In order to prove global-in-time existence of the solution, we use Assumption~\ref{ass2} 3) to avoid the quadratic terms appearing by the composition of a controlled rough path with a smooth function, as stated in Lemma~\ref{lem:composition}. To this aim, we derive the following key result, recalling that $G=A_{-\sigma} NF$. %\textcolor{blue}{Final aim is to get the estimate for $(A_{-\sigma}NF(y), A_{-\sigma} N (DF(y)\circ y' )= (G(y), A_{-\sigma} N (DF(y)\circ G(y)). $ Our solution is $(y,G(y))$ but in order to define the rough integral we don't work with $(G(y), DG(y)G(y))$ as in~\cite{GHN,HN21} but we need $(G(y),A_{-\sigma} N (DF(y)\circ G(y))) $. }
        \begin{lemma}\label{no:q}
            Let Assumption \ref{ass2} be satisfied and further assume that $(y,G(y))\in \cD^{2\gamma}_{X,-\eta}$. Then $(G(y), DG(y)\circ G(y))\in\cD^{2\gamma}_{X,-\eta}$ and the following bound is valid
            \begin{align*}
                \norm{G(y), DG(y)\circ G(y)}_{X,2\gamma,-\eta}\lesssim 1+\norm{y,G(y)}_{X,2\gamma,-\eta}.
            \end{align*}
        \end{lemma}
        \begin{proof}
%        Since $A_{-\sigma}$ and $N$ are linear operators we immediately get that $DG(y)=A_{-\sigma} N \circ DF(y)$ and consequently
  %  \begin{align}\label{eq:ProofGlobal}
                %DG(y)\circ G(y)= A_{-\sigma}N(DF(y) \circ G(y)).
    %\end{align}
    %Similar to the proof of Theorem \ref{local} one can see that also the expression in \ref{eq:ProofGlobal} is well-defined. Now, 
            By definition we have
            \begin{align*}
                 \norm{G(y), DG(y)\circ G(y)}_{X,2\gamma,-\eta}&= \norm{G(y)}_{\infty, -\eta}+ \norm{DG(y)\circ G(y)}_{\infty, -\eta-\gamma}+ \left[DG(y)\circ G(y)\right]_{\gamma,-\eta-2\gamma}\\
                 &+  \left[R^{G(y)}\right]_{\gamma,-\eta-\gamma}+\left[R^{G(y)}\right]_{2\gamma,-\eta-2\gamma},
            \end{align*}
            where $R^{G(y)}_{t,s}:=G(y_t)-G(y_s)-\left(DG(y_s)\circ G(y_s)\right)X_{t,s}$ is  the remainder. The first two terms can be estimated directly, since it is assumed that $F$ has bounded derivatives. Obviously, this holds also for $G$. Therefore, we have
            \begin{align*}
                \norm{G(y)}_{\infty, -\eta}&\leq \abs{G(y_0)}_{-\eta}+T^\gamma \left[G(y)\right]_{\gamma, -\eta}\lesssim 1+\left[y\right]_{\gamma,-\eta}\lesssim 1+\norm{y,G(y)}_{X,2\gamma, -\eta},\\ 
                \norm{DG(y)\circ G(y)}_{\infty, -\eta-\gamma}&\lesssim \norm{DG(y)}_{\infty, \mathcal{L}(\cB_{-\eta-\gamma})}\norm{G(y)}_{\infty,-\eta-\gamma}\lesssim \norm{y,G(y)}_{X,2\gamma, -\eta}.
            \end{align*}
            For the third term we recall that for $\alpha:=\delta_2-\eta>1+\frac{1}{p}$, the Neumann operator maps $\widetilde{\cB}_{\alpha-\gamma}$ to $\cB_{\varepsilon-\gamma}$. Together with \ref{ass2} 3), Remark \ref{rem:Extrapolation} %\textcolor{blue}{(for the first inequality maybe more precise, i.e. we use $A_{-\eta}$ and so on ) }
            and \ref{rem:CRPy} we have
            \begin{align*}
                \left[DG(y)\circ G(y)\right]_{\gamma,\cB_{-\eta-2\gamma}}&=\left[A_{-\sigma}N(DF(y)\circ G(y))\right]_{\gamma,\cB_{-\eta-2\gamma}}=\left[A_{-\eta-2\gamma}N(DF(y)\circ G(y))\right]_{\gamma,\cB_{-\eta-2\gamma}}\\
                &\lesssim \left[N(DF(y)\circ G(y))\right]_{\gamma,\cB_{\varepsilon-2\gamma}}\lesssim \left[N(DF(y)\circ G(y))\right]_{\gamma,\cB_{\varepsilon-\gamma}}\\ 
                &\lesssim \left[DF(y)\circ G(y)\right]_{\gamma,\widetilde{\cB}_{\alpha-\gamma}}\lesssim \left[y\right]_{\gamma,\cB_{-\eta-\gamma}} \lesssim \norm{y,G(y)}_{X,2\gamma,\cB_{-\eta}}.
            \end{align*}
            Consequently we only have to estimate the remainder. For $0\leq s< t \leq T$ we write
            \begin{align*}
                R^{G(y)}_{t,s}=\int_0^1 (DG(y_s+ry_{t,s})-DG(y_s))~\txtd r~G(y_s) X_{t,s}+\int_0^1 DG(y_s+ry_{t,s})~\txtd r~ R^y_{t,s},
            \end{align*}
            and use again the boundedness of the derivatives to obtain
            \begin{align*}
                \left[R^{G(y)}\right]_{\gamma,-\eta-\gamma}&\lesssim\norm{DG(y)}_{\infty,\cL(\cB_{-\eta-\gamma})}(\norm{G(y)}_{\infty,-\eta-\gamma}\left[X\right]_\gamma+\left[R^y\right]_{\gamma,-\eta-\gamma})\\ 
                &\lesssim \norm{y,G(y)}_{X,2\gamma,-\eta}.
            \end{align*}
            For the second remainder term, note that due to the boundedness of $D\left[DF(\cdot)\circ G(\cdot)\right]$, we get the Lipschitz type estimate
            \begin{align*}
                \abs{(DF(x_1)-DF(x_2))G(x_1)}_{\widetilde{\cB}_{-\eta-\gamma}}&\leq \abs{DF(x_1)G(x_1)-DF(x_2)G(x_2)}_{\widetilde{\cB}_{-\eta-\gamma}}\\&
                +\abs{DF(x_2)(G(x_1)-G(x_2))}_{\widetilde{\cB}_{-\eta-\gamma}}\lesssim \abs{x_1-x_2}_{\cB_{-\eta-\gamma}},
            \end{align*}
            for $x_1,x_2\in \cB_{-\eta-\gamma}$. With this property, we get similarly to the first remainder term regarding that $N\in \cL(\widetilde{\cB}_{\alpha-2\gamma},\cB_{\varepsilon-2\gamma})$
            \begin{align*}
                 \abs{R^{G(y)}_{t,s}}_{\cB_{-\eta-2\gamma}}&\lesssim\int_0^1 \abs{(DG(y_s+ry_{t,s})-DG(y_s))G(y_s)}_{\cB_{-\eta-2\gamma}}~\txtd r~ \left[X\right]_{\gamma}\abs{t-s}^\gamma\\ &+\int_0^1 \abs{DG(y_s+ry_{t,s})R^y_{t,s}}_{\cB_{-\eta-2\gamma}}~\txtd r\\ 
                 &\lesssim\int_0^1 \abs{A_{-\sigma}N((DF(y_s+ry_{t,s})-DF(y_s))G(y_s))}_{\cB_{-\eta-2\gamma}}~\txtd r \abs{t-s}^{\gamma}\\&+\int_0^1 \abs{A_{-\sigma}N(DF(y_s+ry_{t,s}))R^y_{t,s}}_{\cB_{-\eta-2\gamma}}~\txtd r\\ 
                 &\lesssim\int_0^1 \abs{(DF(y_s+ry_{t,s})-DF(y_s))G(y_s)}_{\widetilde{\cB}_{\alpha-2\gamma}}~\txtd r \abs{t-s}^\gamma \\ &+\int_0^1 \abs{DF(y_s+ry_{t,s})R^y_{t,s}}_{\widetilde{\cB}_{\alpha-2\gamma}}~\txtd r\\
                 &\lesssim\int_0^1 \abs{(DF(y_s+ry_{t,s})-DF(y_s))G(y_s)}_{\widetilde{\cB}_{\alpha-\gamma}}~\txtd r \abs{t-s}^\gamma\\ &+\norm{DF(y)}_{\infty,\mathcal{L}(\cB_{-\eta-2\gamma};\widetilde{\cB}_{\alpha-2\gamma})}\left[R^y_{t,s}\right]_{2\gamma,\cB_{-\eta-2\gamma}} \abs{t-s}^{2\gamma}\\ 
                 &\lesssim \abs{y_t-y_s}_{\cB_{\eta-\gamma}}\abs{t-s}^\gamma+\norm{y,G(y)}_{X,2\gamma,\cB_{-\eta}}\abs{t-s}^{2\gamma}\\ 
                 &\lesssim \left[y\right]_{\gamma,\cB_{-\eta-\gamma}}\abs{t-s}^{2\gamma}+\norm{y,G(y)}_{X,2\gamma,\cB_{-\eta}}\abs{t-s}^{2\gamma}
                 \lesssim \norm{y,G(y)}_{X,2\gamma,\cB_{-\eta}}\abs{t-s}^{2\gamma},
            \end{align*}
            where we used Remark \ref{rem:CRPy} in the last inequality. This means that $R^{G(y)}\in C^{2\gamma}(\cB_{-\eta-2\gamma})$. Collecting all the estimates proves the statement.
            %\textcolor{blue}{Be more precise in the above estimate: $G$ is not well-defined on $\cB_{\varepsilon-2\gamma}$.}
            \qed
        \end{proof}\\
        
     Based on Lemma~\ref{no:q} we can derive an estimate for the solution of~\eqref{eq:introNew} which does not contain quadratic terms.
    \begin{corollary}\label{bound:int}
        Let $F$ satisfy the Assumption \ref{ass2} and let $(y,G(y))\in \cD^{2\gamma}_{X,-\eta}$. %If we set $G:=A_{-\sigma}NF$ we already did that
        Then we have
        \begin{align*}
            (z,z^\prime):=\left(\int_0^\cdot S_{\cdot-r} G(y_r)~\txtd \X_r, G(y)\right)\in \cD^{2\gamma}_{X,-\eta}
        \end{align*}
        and the following estimate holds true
        \begin{align*}
            \norm{z,z^\prime}_{X,2\gamma,-\eta}\lesssim \abs{G(y_0)}_{-\eta}+\abs{DG(y_0)\circ G(y_0)}_{-\eta-\gamma}+ T^\gamma \norm{G(y), DG(y)\circ G(y)}_{X,2\gamma,-\eta}.
        \end{align*}
    \end{corollary}
     \begin{proof} The claim follows by Corollary \ref{cor:higherreg} combined with the fact that $(G(y), DG(y)\circ G(y))\in \cD^{2\gamma}_{X,-\eta}$ as established in Lemma \ref{no:q}. 
        \qed
    \end{proof}
    \begin{corollary}\label{bound:s}
        Let $f$ and $F$ satisfy Assumptions \ref{ass2} 3) and let $(y,G(y))\in \cD^{2\gamma}_{X,-\eta}$ be the solution of~\eqref{eq:introNew} on a time interval $[0,T]$ with initial data $y_0\in\cB_{-\eta}$. Then we have
        \begin{align*}
        \|y,G(y)\|_{X,2\gamma,-\eta} \lesssim 1 + |y_0|_{-\eta} + T^{1-\delta_1} \|y,G(y)\|_{X,2\gamma,-\eta}.
        \end{align*}
        
% \textcolor{blue}{x should depend on $\gamma,\delta_2$ since our rough convolution does not gain any sp. reg.}\\
% \textcolor{blue}{I think that $x$ has to be equal $\gamma$, check later, i think so too.}
    \end{corollary}
    \begin{proof}
        Since the path component of $(y,G(y))$ solves~\eqref{eq:introNew}, we have
        \begin{align*}
            \norm{y, G(y)}_{X,2\gamma,-\eta}&\leq \norm{S_\cdot y_0,0}_{X,2\gamma,-\eta}+\norm{\int_0^\cdot S_{\cdot-r} f(y_r)~\txtd r,0}_{X,2\gamma,-\eta}\\ 
            &+\norm{\int_0^\cdot S_{\cdot-r} A_{-\sigma}NF(y_r)~\txtd \X_r, A_{-\sigma}NF(y_r)}_{X,2\gamma,-\eta}.
        \end{align*}
        Here it is easily to see that  $(S_\cdot y_0,0)\in \cD^{2\gamma}_{X,-\eta}$ and $(\int_0^\cdot S_{\cdot-r} f(y_r)~\txtd r,0)\in \cD^{2\gamma}_{X,-\eta}$ satisfy the bounds
        \begin{align*}
         \norm{S_\cdot y_0,0}_{X,2\gamma,-\eta}\lesssim \abs{y_0}_{-\eta}, \quad \norm{\int_0^\cdot S_{\cdot-r} f(y_r)~\txtd r,0}_{X,2\gamma,-\eta}\lesssim (1+\norm{y}_{\infty, -\eta}) T^{1-\delta_1},
        \end{align*}
       using Assumption~\ref{ass2} 1). %on $f$ for the second term.
       % as in \cite[Lemma 3.3-3.4]{HN21}.
        % \textcolor{blue}{Or should this be a separate lemma? This is ok, the estimate for the stochastic convolution should be a separate lemma.} 
        % If we use again the notation $G:=A_{-\sigma}NF$, it can be shown, like in \cite[Lemma 3.5]{HN21}, that the rough \textcolor{blue}{we should write this lemma. apart from this everything should be fine, will check it today} convolution has the bound
        % \begin{align*}
        %     \norm{\int_0^\cdot S_{\cdot-r} G(y_r)~\txtd \X_r, G(y_r)}_{X,2\gamma,-\eta}&\lesssim \abs{G(y_0)}_{-\eta}+\abs{DG(y_0)\circ G(y_0)}_{-\eta}\\ 
        %     &+ T^\gamma \norm{G(y), DG(y)\circ G(y)}_{X,2\gamma,-\eta}.
        % \end{align*}
        If we now combine these estimates with those obtained in Lemma \ref{no:q} and Lemma \ref{bound:int}, we get 
        \begin{align*}
            \norm{y, G(y)}_{X,2\gamma,-\eta}&\lesssim \abs{y_0}_{-\eta}+(1+\norm{y}_{\infty, -\eta}) T^{1-\delta_1}\\
            &+\abs{G(y_0)}_{-\eta}+\abs{DG(y_0)\circ G(y_0)}_{-\eta-\gamma}+ T^\gamma \norm{G(y), DG(y)\circ G(y)}_{X,2\gamma,-\eta}\\ 
            &\lesssim 1+ \abs{y_0}_{-\eta} +T^{1-\delta_1}\norm{y,G(y)}_{X,2\gamma,-\eta} +T^{\gamma}\norm{y,G(y)}_{X,2\gamma,-\eta},
        \end{align*}
        where also the boundedness of the derivative $DG$ was used. Since $T\leq 1$ and $1-\delta_1< \gamma$, due to $\delta_1\geq 2\gamma$, we have $T^\gamma<T^{1-\delta_1}$ which leads to the assertion.
        \qed
    \end{proof}\\
    
    Putting the previous deliberations together, we finally obtain the existence of a global-in-time solution.
    
        \begin{theorem}\em{(Existence of a global-in-time solution for~\eqref{eq:introNew})}\label{thm:global}
            Assume that $f$ and $F$ satisfy the Assumptions \ref{ass2} 3). Then there exists for every initial condition $y_0\in \cB_{-\eta}$ a unique solution $(y,A_{-\sigma}NF(y))=(y,G(y))\in \cD^{2\gamma}_{X,-\eta}\left([0,T]\right)$ to \eqref{eq:introNew}) such that $y$ fulfils \eqref{eq:mildSolution} for all $t\in [0,T]$.
        \end{theorem}
        \begin{proof}
         Due to Corollary~\ref{bound:s} we obtain by a standard concatenation argument (compare~\cite[Lemma 3.8]{HN21}) the following a-priori bound for the path component of the solution on a time interval $[0,T]$ for an arbitrary $T>0$. Namely, there exist constants $M_1,M_2>0$ such that 
         \[
         \|y\|_{\infty,-\eta,[0,T]}\leq M_1 r e^{M_2T},
         \]
        where $r:=1\vee|y_0|_{-\eta}$. Therefore, the local solution obtained in Theorem~\ref{local} cannot exhibit a finite-time blow-up as justified in~\cite[Theorem 3.9]{HN21}. 
            \qed 
        \end{proof}\\
      Furthermore, we point out an advantage of the rough path approach, which gives stability of the solution with respect to the initial condition and the noise term in contrast to the It\^o calculus. This is an immediate consequence of the pathwise construction of the solutions of~\eqref{eq:intro}. To this aim, we let $\widetilde{\X}=(\widetilde{X},\widetilde{\XX})$ be another $\gamma$-H\"older rough path with $\widetilde{X}_0=0$ and consider two controlled rough paths $(u,u^\prime)\in \cD^{2\gamma}_{X,-\eta}\left([0,T]\right)$ respectively $(v,v^\prime)\in \cD^{2\gamma}_{\widetilde{X},-\eta}\left([0,T]\right)$
      and define for $0\leq \gamma^\prime <\gamma$ the metric
        \begin{align}\label{def:distance}
            \begin{split}
            d_{\gamma^\prime,\gamma,-\eta}(u,u^\prime,v,v^\prime)=\norm{u-v}_{\infty,-\eta}&+\norm{u^\prime-v^\prime}_{\infty,-\eta-\gamma}+[u^\prime-v^\prime]_{\gamma^\prime,-\eta-2\gamma}\\
            &+[R^u-R^v]_{\gamma^\prime,-\eta-\gamma}+[R^u-R^v]_{2\gamma^\prime,-\eta-2\gamma}.
            \end{split}
        \end{align}
       The dependence of $d$ on $\X$ and $\widetilde{\X}$ is not displayed here for notational simplicity. For our aims, we first state stability results for the rough integration and composition with smooth functions.
        \begin{lemma}\label{lem:stability}
            \begin{itemize}
                \item[i)] Let $(y,y^\prime)\in \cD^{2\gamma}_{X,-\eta}\left([0,T]\right)$ and $(\widetilde{y},\widetilde{y}^\prime)\in \cD^{2\gamma}_{\widetilde{X},-\eta}\left([0,T]\right)$. If $\rho_\gamma(\X),\rho_\gamma(\widetilde{\X}), \norm{y,y^\prime}_{X,2\gamma,-\eta}, \norm{\widetilde{y},\widetilde{y}^\prime}_{\widetilde{X},2\gamma,-\eta}$ are bounded by the same constant, then for every $\frac{1}{3}< \gamma^\prime\leq\gamma$ we have
                \begin{align}\label{est:Stability}
                    d_{\gamma^\prime,\gamma,-\eta}(z ,z^\prime,\widetilde{z},\widetilde{z}^\prime)\lesssim d_{\gamma,[0,T]}(\X,\widetilde{\X})+\abs{y_0-\widetilde{y}_0}_{-\eta}+d_{\gamma^\prime,\gamma,-\eta}(y,y^\prime,\widetilde{y},\widetilde{y}^\prime) T^{\gamma-\gamma^\prime},
                \end{align}
                where $z_t:=\int_0^t S_{t-s} y_s~\txtd \X_s,~ z^\prime:=y$ and analogously  $\widetilde{z}_t:=\int_0^t S_{t-s} \widetilde{y}_s~\txtd \widetilde{\X}_s,~ \widetilde{z}^\prime:=\widetilde{y}$. 
                \item[ii)] In addition to the assumptions in i) we suppose that $F$ satisfies Assumption \ref{ass2} 2). Then for every $\frac{1}{3}< \gamma^\prime\leq\gamma$ we have
                \begin{align*}
                    d_{\gamma^\prime,\gamma,-\eta}(z ,z^\prime,\widetilde{z},\widetilde{z}^\prime)\lesssim d_{\gamma,[0,T]}(\X,\widetilde{\X})+\abs{y_0-\widetilde{y}_0}_{-\eta}+d_{\gamma^\prime,\gamma,-\eta}(y,y^\prime,\widetilde{y},\widetilde{y}^\prime) T^{\gamma-\gamma^\prime},
                \end{align*}
                where $z_t:=F(y_t), z^\prime_t:=DF(y_t)\circ y^\prime_t$ and analogously $\widetilde{z}_t:=F(\widetilde{y}_t),\widetilde{z}^\prime_t:=DF(\widetilde{y}_t)\circ \widetilde{y}^\prime_t$.
            \end{itemize}
        \end{lemma}
        \begin{proof}
            \begin{itemize}
                \item[i)] The idea of the proof is to analyze the difference of the stochastic convolutions $z-\widetilde{z}$. As in the proof of Lemma~\ref{lemma:3.7} one needs an approximating sequence $\Xi:=\xi-\widetilde{\xi}$, where $\xi_{t,s}:=y_sX_{t,s} +y_s^\prime \XX_{t,s}$ and $\widetilde{\xi}_{t,s}:=\widetilde{y}_s\widetilde{X}_{t,s} +\widetilde{y}_s^\prime \widetilde{\XX}_{t,s}$ are the approximations of the individual integrals $z$ and $\widetilde{z}$. %By adding $\widetilde{y}_s X_{t,s}-\widetilde{y}_s X_{t,s}$ and $\widetilde{y}_s^\prime \XX_{t,s}-\widetilde{y}_s^\prime \XX_{t,s}$ we get 
                This reads as
                \begin{align}\label{eq:Appr}
                    \Xi_{t,s}=(y_s-\widetilde{y}_s)X_{t,s}+\widetilde{y}_s(X_{t,s}-\widetilde{X}_{t,s})+(y^\prime_{s}-\widetilde{y}^\prime_s)\XX_{t,s}+\widetilde{y}^\prime_s (\XX_{t,s}-\widetilde{\XX}_{t,s}).
                \end{align}
                 Combining now \eqref{eq:Appr} with the proofs of Lemma~\ref{lemma:3.7} and Corollary \ref{cor:RoughConv} we derive
                \begin{align*}
                    \abs{z_{t,s}-\widetilde{z}_{t,s}-S_{t-s}\Xi_{t,s}}_{-\eta-i\gamma}\lesssim (d_{\gamma,[0,T]}(\X,\widetilde{\X})+ d_{\gamma^\prime,\gamma,-\eta}(y,y^\prime,\widetilde{y},\widetilde{y}^\prime)T^{\gamma^\prime})\abs{t-s}^{i\gamma^\prime},
                \end{align*}
                for $i=0,1,2$. To get now \eqref{est:Stability} we have to estimate the individual terms of the distance \eqref{def:distance} similar to Theorem \ref{thm:RoughConv}, see also \cite[Lemma 3.13]{GHairer}. For example, we have 
                \begin{align*}
                    \abs{z_t-\widetilde{z_t}}_{-\eta}&\leq \abs{z_{t,0}-\widetilde{z}_{t,0}-S_{t}\Xi_{t,0}}_{-\eta}+
                    \abs{S_{t}\Xi_{t,0}}_{-\eta}\\
                    &\lesssim d_{\gamma,[0,T]}(\X,\widetilde{\X})+ d_{\gamma^\prime,\gamma,-\eta}(y,y^\prime,\widetilde{y},\widetilde{y}^\prime)T^{\gamma^\prime}+T^\gamma \norm{y-\widetilde{y}}_{\infty,-\eta}\rho_\gamma(\X)\\ &+T^\gamma\norm{\widetilde{y}}_{\infty,-\eta}d_{\gamma,[0,T]}(\X,\widetilde{\X})+\norm{y^\prime-\widetilde{y}^\prime}_{\infty,-\eta}\rho_\gamma(\X)+\norm{\widetilde{y}^\prime}_{\infty,-\eta}d_{\gamma,[0,T]}(\X,\widetilde{\X})\\
                        &\lesssim d_{\gamma,[0,T]}(\X,\widetilde{\X})+ d_{\gamma^\prime,\gamma,-\eta}(y,y^\prime,\widetilde{y},\widetilde{y}^\prime)T^{\gamma^\prime}.
                \end{align*}
                The remaining terms of \eqref{def:distance} can be handled analogously. 
                \item[ii)] The statement can be obtained following the steps of the proof of Lemma \ref{lem:composition}, see also \cite[Lemma 3.14]{GHairer}.
            \end{itemize}
            \qed
        \end{proof}\\
      Based on this result we can establish the continuous dependence of the solution with respect to the noise and initial data.
        \begin{theorem}\em{(Stability of the solution)}\label{thm:stability} 
            Assume that $f$ and $F$ satisfy the Assumptions \ref{ass2} 3). Let $y_0,\widetilde{y}_0 \in \cB_{-\eta}$ be two initial conditions  and let $(y,G(y))\in \cD^{2\gamma}_{X,-\eta}\left([0,T]\right), (\widetilde{y},G(\widetilde{y}))\in \cD^{2\gamma}_{\widetilde{X},-\eta}\left([0,T]\right)$ be the solutions of \eqref{eq:introNew} driven by $\X$ respectively $\widetilde{\X}$ with initial conditions $y_0$ and $\widetilde{y}_0$. If $\rho_\gamma(\X),\rho_\gamma(\widetilde{\X}),\norm{y_0}_{-\eta},\norm{\widetilde{y}_0}_{-\eta}$ are bounded by the same constant, then for every $\frac{1}{3}< \gamma^\prime<\gamma$ we have
            \begin{align*}
                d_{\gamma^\prime,\gamma,-\eta}(y,G(y),\widetilde{y},G(\widetilde{y}))\lesssim d_{\gamma,[0,T]}(\X,\widetilde{\X})+\abs{y_0-\widetilde{y}_0}_{-\eta}.
            \end{align*}
        \end{theorem}
        \begin{proof}
          Using Lemma~\ref{lem:stability} one infers that
            \begin{align*}
                d_{\gamma^\prime,\gamma,-\eta}(y,G(y),\widetilde{y},G(\widetilde{y}))\lesssim d_{\gamma^\prime,\gamma,-\eta}(y,G(y),\widetilde{y},G(\widetilde{y})) T^\kappa+ d_{\gamma,[0,T]}(\X,\widetilde{\X})+\abs{y_0-\widetilde{y}_0}_{-\eta},
            \end{align*}
            where $0<\kappa:=\gamma-\gamma^\prime$. Choosing $T$ small enough proves the statement.
            \qed
        \end{proof}
        
        % \textcolor{blue}{Maybe a few details on the proof, take a look if there is any lower bound on $\gamma^\prime$\\
        % yes and why we need $\gamma'$. The steps of the proof are similar to~\cite[Thm 4.5]{GHairer}. We could also add the equivalence of the mild and variational solution~\cite{VeraarSchnaubelt}..or maybe it's better to postpone it for another work}
        % \textcolor{blue}{I will check it this week }
       
        \paragraph{The Young case.}
        For the sake of completeness, we now consider Dirichlet boundary noise in the Young regime, i.e.~if the random input~$X\in C^{\widetilde{\gamma}}(\R)$ for $\widetilde{\gamma}\in(\frac{1}{2},1)$. We denote, as in Remark \ref{rem:dirichlet}, by $\mathfrak{D}$ the solution operator of~\eqref{eq:bvp} with $\widetilde{\cC}=\gamma_\partial$. In this case, the domain $D(A)$ is different now, which means that the extrapolation-interpolation scale according to $A$ changes. To point that out, we denote the extrapolation spaces by $\cB_\beta^\mathfrak{D}$ and $\widetilde{\cB}^\mathfrak{D}_{\beta}:=B^{\beta-\frac{1}{p}}_{p,p}(\partial \cO)$. In this case the spaces $\cB^{\mathfrak{D}}_\beta$ are given by
        \begin{align*}
\cB^\mathfrak{D}_{\frac{\beta}{2}}:=H^{\beta,p}_{\widetilde{\cC}}(\cO):=
            \begin{cases}
                \{ u\in H^{\beta,p}(\cO) : \gamma_\partial u =0 \},& \beta>\frac{1}{p}\\
                H^{\beta,p}(\cO) ,& -2+\frac{1}{p}<\beta<\frac{1}{p}.
            \end{cases}
        \end{align*}
        Then, as justified in Remark \ref{rem:dirichlet}, the Dirichlet operator $\mathfrak{D}$ is bounded from $\widetilde{\cB}_\beta^\mathfrak{D}$ to $\cB^\mathfrak{D}_{\varepsilon_\mathfrak{D}}$ with $\varepsilon_\mathfrak{D}<\frac{1}{2p}$. Furthermore, the boundary value problem~\eqref{eq:bvp} (with $\cC$ replaced by $\widetilde{\cC}$) has a strong solution for $\beta>\frac{1}{p}$. 
        However, for the definition of Young's integral~\eqref{young} we need to consider paths which are continuous in $\cB^\mathfrak{D}_\beta$ and $\widetilde{\gamma}$-H\"older continuous with values in $\cB^\mathfrak{D}_{\beta-\widetilde{\gamma}}$. This means that the index $\beta-\widetilde{\gamma}$  can become negative if we only assume that $\beta>\frac{1}{p}$. In this case the theory of the interpolation-extrapolation scale in \cite{Amann} breaks down. To overcome this issue, we additionally assume that $\beta>1+\frac{1}{p}$. Then  $\beta-\widetilde{\gamma}>\frac{1}{p}$ and $\mathfrak{D}:\widetilde{\cB}^\mathfrak{D}_{\beta-\widetilde{\gamma}} \to \cB^\mathfrak{D}_{\varepsilon_\mathfrak{D}-\widetilde{\gamma}}$. Furthermore, just as in the Neumann case, the condition $\widetilde{\gamma}>1-\varepsilon_\mathfrak{D}$ also needs to be satisfied. In the Neumann case, this condition automatically holds for rough noise, i.e.~$\gamma\in (\frac{1}{3},\frac{1}{2}]$. For Dirichlet boundary noise in the Young regime, this leads to an additional restriction on the regularity of the noise, compare~\cite{DuncanMaslowski}. Therefore we choose $\widetilde{\gamma}\in (1-\frac{1}{2p},1)$.  Under these assumptions, we show that it is possible to incorporate Dirichlet boundary noise in~\eqref{eq:intro}. This SPDE can be now rewritten as
        
        % If we consider now the Young case, i.e.~$X\in C^{\widetilde{\gamma}}(\R)$ for $\widetilde{\gamma}\in(\frac{1}{2},1)$, we get $\varepsilon_\mathfrak{D}:=\frac{1}{2p}-\widetilde{\delta}$ with $\widetilde{\delta}>0$ small. This means that we can choose $\widetilde{\gamma}:=\frac{5}{6}+ \hat{\delta}<1$ with $\hat{\delta}>\widetilde{\delta}$ to satisfy $1-\varepsilon_\mathfrak{D}<\widetilde{\gamma}$. 
        %  Replacing now the Neumann boundary conditions $\mathcal{C}$ with $\widetilde{\cC} u:=\gamma_\partial u$ and regarding that $\widetilde{\gamma}\in(\frac{1}{2},1)$, the equation~\eqref{eq:intro} rewrites as
        \begin{align}\label{eq:introNewDir}
            \begin{cases}
                  \txtd y = \left(Ay + f(y)\right)~\txtd t + A_{-\sigma_\mathfrak{D}}\mathfrak{D}F(y)~\txtd X_t,\\
                  y(0)=y_0\in \cB^\mathfrak{D}_{-\eta_\mathfrak{D}},
            \end{cases}
        \end{align}
         where $\eta_\mathfrak{D}:=1-\varepsilon_\mathfrak{D}$ and $\sigma_\mathfrak{D}:=\eta_\mathfrak{D}-\tilde{\gamma}$. 
          \begin{assumptions}(Young case)\label{ass3}
            There exists $\delta_2> \eta_\mathfrak{D}+1+\frac{1}{p}$ such that for any $\vartheta\in \{0,\tilde{\gamma} \}$ the diffusion term $F:\cB^\mathfrak{D}_{-\eta_\mathfrak{D}-\vartheta}\to \widetilde{\cB}^\mathfrak{D}_{-\eta_\mathfrak{D}-\vartheta+\delta_2}$ is two times continuously Fr\'{e}chet differentiable with bounded derivatives.
        \end{assumptions}
         Based on the arguments of Theorem~\ref{thm:global} we derive. 
%        \textcolor{blue}{ i stated the restriction  $\varepsilon_\mathfrak{D}>1-\tilde{\gamma}$ in the next thm and changed $\widetilde{\gamma}\in(1/2,1)$. for additive noise we get for $p=2$ that $\widetilde{\gamma}\in(3/4,1)$}
        \begin{theorem}\label{globalDirichlet}{\em (Dirichlet boundary noise in the Young case)}
            Let $X\in C^{\tilde{\gamma}} (\R)$ with $\tilde{\gamma}\in (1-\frac{1}{2p},1)$. Assume that $f$ and $F$ satisfy Assumption \ref{ass2} 3) replacing \ref{ass2} 2) with \ref{ass3}. Then there exists for every initial condition $y_0\in \cB^\mathfrak{D}_{-\eta_\mathfrak{D}}$ a unique mild solution $y\in C(\cB^\mathfrak{D}_{-\eta_{\mathfrak{D}}})\cap C^{\widetilde{\gamma}}(\cB^\mathfrak{D}_{-\eta_{\mathfrak{D}}-\widetilde{\gamma}})$ that satisfies  for all $t\in [0,T]$
            $$ y_t=S_t y_0 + \int_0^t S_{t-r} f(y_r)~\txtd r+\int_0^t S_{t-r} A_{-\sigma}\mathfrak{D}F(y_r)~\txtd X_r,$$
where the integral is understood in the Young sense~\eqref{young}.

            %$(y,A_{-\sigma_\mathfrak{D}}\mathfrak{D}F(y))\in\cD^{2\tilde{\gamma}}_{-\eta_\mathfrak{D},X}([0,T])$ to \eqref{eq:introNewDir}. %\textcolor{blue}{here we have only the path component $y$, we don't need the Gubinelli derivative. We have to write the mild solution for $y$ and say to which space it belongs}
        \end{theorem}
        \begin{remark}
         The theory developed in this work can be extended to time-dependent operators $A(t)$ generating parabolic evolution families $(U(t,s))_{t\geq s}$ as in~\cite{VeraarSchnaubelt}. Analogously, such operators satisfy for $t>s$ similar estimates to~\eqref{hg:1} and~\eqref{hg:2}, i.e. 
            \begin{align}
                |(U(t,s)-\Id) x|_{\alpha}&\lesssim |t-s|^\sigma |x|_{\alpha+\sigma},\\
                |U(t,s)x|_{\alpha+\sigma}&\lesssim |t-s|^{-\sigma}|x|_{\alpha}.
            \end{align}
            However the extrapolation scale has to be extended to cover this setting as well. Moreover, we believe that  this theory can be extended to higher-order differential operators, since our method is independent of the order of the operator. The only necessary ingredient is the existence of a continuous solution operator equivalent to the Neumann / Dirichlet operator. 
            % $N$.
             
             %In this case, one needs an appropriate solution operator $N$ as introduced in \cite{Roitberg}.  
        %     \textcolor{blue}{In the same way, the theory should work for higher order operators. Elementary for this is that there is an equivalent to the solution operator $N$. See for example the theory of Roitberg \cite{Roitberg}, which considers operators of arbitrary order. The rest of the theory is otherwise independent of the order.}
           %  \textcolor{blue}{should we say something about higher order differential operators?}
        \end{remark}

    \section{Random dynamical systems}\label{sec:rds}
        
        Based on our global well-posedness result, under suitable assumptions on the driving
        rough path, we are able to construct a random dynamical system corresponding to~\eqref{eq:intro}. To this aim, we introduce some concepts from the theory of random dynamical systems~\cite{Arnold}.
        \begin{definition}\label{mds} 
            Let $(\Omega,\mathcal{F},\mathbb{P})$ stand for a probability space and 
            $\theta:\mathbb{R}\times\Omega\rightarrow\Omega$ be a family of 
            $\mathbb{P}$-preserving transformations (i.e.~$\theta_{t}\mathbb{P}=
            \mathbb{P}$ for $t\in\mathbb{R}$) having the following properties:
            \begin{description}
                \item[(i)] The mapping $(t,\omega)\mapsto\theta_{t}\omega$ is 
                $(\mathcal{B}(\mathbb{R})\otimes\mathcal{F},\mathcal{F})$-measurable, where 
                $\mathcal{B}(\cdot)$ denotes the Borel $\sigma$-algebra;
                \item[(ii)] $\theta_{0}=\textnormal{Id}_{\Omega}$;
                \item[(iii)] $\theta_{t+s}=\theta_{t}\circ\theta_{s}$ for all 
                $t,s,\in\mathbb{R}$.
            \end{description}
            Then the quadrupel $(\Omega,\mathcal{F},\mathbb{P},(\theta_{t})_{t\in\mathbb{R}})$ 
            is called a metric dynamical system.
        \end{definition}

        \begin{definition}
            \label{rds} 
            A continuous random dynamical system on a separable Banach space $\cX$ over a metric dynamical 
            system $(\Omega,\mathcal{F},\mathbb{P},(\theta_{t})_{t\in\mathbb{R}})$ 
            is a mapping $$\varphi:[0,\infty)\times\Omega\times \cX\to \cX,
            \mbox{  } (t,\omega,x)\mapsto \varphi(t,\omega,x), $$
            which is $(\mathcal{B}([0,\infty))\otimes\mathcal{F}\otimes
            \mathcal{B}(\cX),\mathcal{B}(\cX))$-measurable and satisfies:
            \begin{description}
                \item[(i)] $\varphi(0,\omega,\cdot{})=\textnormal{Id}_{\cX}$ 
                for all $\omega\in\Omega$;
                \item[(ii)]$ \varphi(t+\tau,\omega,x)=
                \varphi(t,\theta_{\tau}\omega,\varphi(\tau,\omega,x)), 
                \mbox{ for all } x\in \cX, ~t,\tau\in[0,\infty),~\omega\in\Omega;$
                \item[(iii)] $\varphi(t,\omega,\cdot{}):\cX\to \cX$ is 
                continuous for all $t\in[0,\infty)$ and all $\omega\in\Omega$.
            \end{description}
        \end{definition}
        
        The second property in Definition~\ref{rds} is referred to as the 
        cocycle property. The generation of a random dynamical system from an It\^{o}-type stochastic partial differential equation (SPDE) has been a long-standing open problem, since Kolmogorov's theorem breaks down for random fields parametrized by infinite-dimensional Banach spaces.~As a consequence it is not known how to obtain a random dynamical system from an SPDE, since its solution is defined almost surely, which contradicts the cocycle property. In particular, this means that there are exceptional sets which depend on the initial condition and it is not clear how to define a random dynamical system if more than countably many exceptional sets occur. This issue does not occur in a pathwise approach.  Provided that global existence of solutions is ensured, rough path driven equations generate random
        dynamical systems if the driving rough path forms a rough path cocycle, as established in~\cite{BRiedelScheutzow}.\\
        
        The next concept describes a model of the driving noise. Let $(\Omega,\mathcal{F},\mathbb{P},(\theta_{t})_{t\in\mathbb{R}})$ be a metric dynamical system as in Definition~\ref{mds}. We say that 
        \begin{align*}
            \mathbf{X}=(X,\xx):\Omega\to C^{\gamma}_{\text{loc}}([0,\infty);\mathbb{R}^d) \times C^{2\gamma}_{\text{loc}}([0,\infty);\mathbb{R}^{d\times d}) 
        \end{align*}
        is a continuous ($\gamma$-H\"older) rough path cocycle~\cite{BRiedelScheutzow} if $\mathbf{X}|_{[0,T]}$ is a continuous $\gamma$-H\"older rough path for every $T>0$ and $\omega\in\Omega$ and the following cocycle property holds true for every $s,t\in[0,\infty)$ and $\omega\in\Omega$
        \begin{align*}
            &X_{s+t,s}(\omega)= X_t(\theta_s\omega),\\
            &\xx_{s+t,s}(\omega)=\xx_{t,0}(\theta_s\omega).
        \end{align*}
        According to~\cite[Section 2]{BRiedelScheutzow} rough path lifts of various stochastic processes define cocycles.
        These include Gaussian processes with stationary increments under certain
        assumption on the covariance function~\cite[Chapter 10]{FrizHairer} and particularly apply to the fractional Brownian motion with Hurst index $H>\frac{1}{4}$.\\
        
        Based on Theorem~\ref{thm:global} we immediately derive the existence of a random dynamical system associated to~\eqref{eq:intro}. Using a classical flow transformation, such a statement together with the existence of a random attractor was obtained for a system of SPDEs  
        with dynamical boundary conditions in~\cite{BDS}.
        
        \begin{theorem}\label{thm:rds}  Let $\mathbf{X}=(X,\xx)$ be a $\gamma$-H\"older rough path cocycle. 
            Under the assumptions of Theorem~\ref{thm:global}, the solution operator of~\eqref{eq:intro} generates a random dynamical system on $\cB_{-\eta}$. 
        \end{theorem}
        \begin{proof}
            In order to verify the cocycle property, regarding Theorem~\ref{thm:global}, we let $t,\tau\in\R_{+}$, recall that $G=A_{-\sigma}NF$  and compute
            \begin{align*}
                \varphi({t+\tau},\omega,y_0):=y_{t+\tau} &=S_{t+\tau}y_0 + \int\limits_{0}^{t+\tau} S_{t+\tau-r} 
                f(y_{r}) ~\txtd r + \int\limits_{0}^{t+\tau} S_{t+\tau-r} 
                G(y_{r}) ~\txtd \mathbf{X}_{r}\\
                & = S_{t} y_{\tau} + \int\limits_{0}^{t} S_{t-r} f(y_{r+\tau}) 
                ~\txtd r + \int\limits_{0}^{t}S_{t-r}G(y_{r+\tau}) ~\txtd 
                \theta_{\tau}\mathbf{X}_{r}.
            \end{align*}
            Note that the shift property of the rough integral~\eqref{Gintegral} is immediate, see \cite[Corollary 4.5]{HN20}. The $(\mathcal{B}([0,\infty))\otimes\mathcal{F}\otimes\mathcal{B}(\cB_{-\eta}), \mathcal{B}(\cB_{-\eta}))$-measurability of $\varphi$ follows by well-known arguments. More precisely, one considers a sequence of (classical) solutions $(y^{n},(y^{n})')_{n\in\mathbb{N}}$ of~\eqref{eq:intro} corresponding to smooth approximations $(X^{n},\mathbb{X}^{n})_{n\in\mathbb{N}}$ of $(X,\mathbb{X})$. Obviously, the mapping $(t,X,\xi)\mapsto y^{n}_{t}$ is $(\mathcal{B}([0,T])\otimes\mathcal{F}\otimes\mathcal{B}(\cB_{-\eta}), \mathcal{B}(\cB_{-\eta}))$-measurable for any $T>0$. Since $y$ continuously depends on the rough input $X=(X,\mathbb{X})$, according to \cite[Lemma 3.12]{GHairer}, one concludes that $\lim\limits_{n\to\infty}y^{n}_t=y_{t}$. This gives the measurability of $y$ with respect to $\mathcal{F}\otimes \mathcal{B}(\cB_{-\eta})$. Due to the time-continuity of $y$, we obtain by~\cite[Chapter~3]{CastaingValadier} the $(\mathcal{B}([0,T])\otimes\mathcal{F}\otimes\mathcal{B}(\cB_{-\eta}), \mathcal{B}(\cB_{-\eta}))$-measurability of the mapping $(t,\omega,\xi)\mapsto y_{t}$ for any $t\geq 0$.
            \qed
        \end{proof}
        
        \begin{remark}
            Naturally, based on the statement of Theorem~\ref{globalDirichlet}, we obtain a random dynamical system in the Young case for the SPDE~\eqref{eq:intro} with multiplicative Dirichlet boundary noise.
        \end{remark}
        
    \section{Examples}\label{examples}
        Here we provide an application of our theory, specifying concrete examples for $A$ and $F$. Since the condition on the drift term $f$ is less restrictive,  examples such as polynomial nonlinearities or Nemytskii type operators are possible. Therefore we focus here on examples for the diffusion coefficient $F$. 
        In both examples, we consider $p=2$ and the formal operators are augmented by either Neumann or Dirichlet boundary conditions
        \begin{align}\label{formalOp}
            \cA u:= \sum_{i,j=1}^d \partial_i \left(a_{ij}\partial_j \right)u +bu,\quad \cC u:= \sum_{i,j=1}^d \nu_i\gamma_\partial a_{ij}\partial _ju, \quad \widetilde{\cC} u = \gamma_\partial u,
        \end{align}
        where the coefficients $a_{ij}, b:\overline{\cO}\to \R$ are smooth,  $(a_{ij})_{i,j=1}^d$ is symmetric and uniform elliptic, meaning that there exists some constant $k>0$ such that for all $\xi \in \R^d$ and $x\in \overline{\cO}$ we have
        \begin{align*}
            \sum_{i,j=1}^d a_{ij}(x)\xi_i \xi_j\geq k \abs{\xi}^2.
        \end{align*}
        In this case $(\cA,\cC)$ and $(\cA,\widetilde{\cC})$ are normally elliptic boundary value problems and the corresponding $L^2(\cO)$-realization $A$, after a possible shift, has bounded imaginary powers \cite[Theorem 2.3]{DDHPV}.% in both cases.
        \begin{example}(Young case and Dirichlet boundary noise)\\
            We consider Dirichlet case $(\cA,\widetilde{\cC})$ and the respective realization $A$. For the regularity of the noise, we take $\tilde{\gamma}\in (\frac{3}{4},1)$, such that the condition $\varepsilon_\mathfrak{D}>1-\widetilde{\gamma}$ is satisfied. In order to verify the Assumption~\ref{ass3}, we first investigate the extrapolation spaces. Recall that for $\beta>-\frac{3}{2}$ they are given by 
            \begin{align*}
            \cB^\mathfrak{D}_{\frac{\beta}{2}}:=
            \begin{cases}
                \{ u\in H^{\beta}(\cO) : \gamma_\partial u =0 \},& \beta>\frac{1}{2}\\
                H^{\beta}(\cO) ,& -\frac{3}{2}<\beta<\frac{1}{2}\\ 
                \{ u\in H^{-\beta}(\cO) : \gamma_\partial u =0 \}^\prime, & \beta<-\frac{3}{2}.
            \end{cases}
            \end{align*}
          Regarding Theorem~\ref{globalDirichlet} and since $-\eta_\mathfrak{D}-\tilde{\gamma}>2\tilde{\gamma}>-2$, we are interested in $\cB^\mathfrak{D}_{-2}$. Now, it is known that $\cB^\mathfrak{D}_{-2}$ is given by the dual space of $D(A^2)$, see \cite[Theorem V.1.5.12]{Amann2}. But since also $H^4_0(\cO)\hookrightarrow D(A^2)$ and $(H^4_0(\cO))^\prime=H^{-4}(\cO)$ we can continuously embed $\cB^\mathfrak{D}_{-2}$ into $H^{-4}(\cO)$\tim{, where $H^4_0(\cO):=H^4(\cO)\cap H^1_0(\cO)$}. Based on the above considerations, it is sufficient to find a linear continuous mapping from $H^{-4}(\cO)$ to $H^{\tilde{\delta}+1}(\partial \cO)=B^{\tilde{\delta}+1}_{2,2}(\partial \cO)$ for a $\tilde{\delta}>0$. An example for this is a slightly adapted lifting operator. More precisely, we set $\nu_\mathfrak{D}:=-\frac{11}{2}-\tilde{\delta}$ and consider
            \begin{align*}
                \Lambda^{\nu_\mathfrak{D}}:H^{-4}(\R^d)\to H^{\tilde{\delta}+1+\frac{1}{2}}(\R^d),f\mapsto\cF^{-1}(1+\abs{\cdot}^2)^{\frac{\nu_\mathfrak{D}}{2}}\cF f,
            \end{align*}
            where $\cF$ is the Fourier transform. Such examples occur in the theory of pseudo-differential operators, see for example \cite[Theorem 18.1.13]{Hor}.  Furthermore, there exist linear and continuous  operators $e_\cO:H^{-4}(\cO)\to H^{-4}(\R^d)$ and  $r_\cO:H^{\tilde{\delta}+1+\frac{1}{2}}(\R^d)\to H^{\tilde{\delta}+1+\frac{1}{2}}(\cO)$. These are called {\em retraction/co-retraction}, see \cite[Theorem 4.2.2]{Triebel78}. Since the trace operator $\gamma_\cO:H^{\tilde{\delta}+1+\frac{1}{2}}(\cO)\to H^{\tilde{\delta}+1}(\partial \cO)$ is also linear and continuous, we conclude that the operator $F:=\gamma_{\partial} r_\cO \Lambda^{\nu_\mathfrak{D}} e_\cO$ fulfils the Assumption \ref{ass2} 2). Moreover, since $F$ is linear and bounded, the same holds true for $G=A_{-\sigma_\mathfrak{D}} \mathfrak{D}F$ and Theorem \ref{globalDirichlet} provides a global-in-time solution in this case.
        \end{example}
        \begin{example}(Rough Neumann boundary noise)\\
        Returning to the rough Neumann boundary noise, we consider $(\cA,\cC)$ and the respective realization $A$. %\textcolor{blue}{ So let $\gamma\in (\frac{1}{3},\frac{1}{2}] $--this always holds for these $\gamma$ such that $\varepsilon>1-\gamma$, $\varepsilon=3/4-\delta$ with a small $\delta$} 
        We recall the characterization of the extrapolation spaces in this case
        \begin{align*}
            \cB_{\frac{\beta}{2}}:=H^\beta_\cC(\cO):=
            \begin{cases}
                \{ u\in H^{\beta}(\cO) : \cC u =0 \},& \beta>\frac{3}{2}\\
                H^{\beta}(\cO) ,& -\frac{1}{2}<\beta<\frac{3}{2},\\ 
                \left(H^{-\beta}(\cO)\right)^\prime, & -\frac{3}{2}<\beta\leq-\frac{1}{2}\\
                \{ u\in H^{-\beta}(\cO) : \cC u =0 \}^\prime, & \beta<-\frac{3}{2}.
            \end{cases}.
        \end{align*}
        Keeping the results obtained in Section~\ref{main} in mind, we are interested in the case $\beta=2(-\eta-2\gamma)$. Since $-\eta >-\gamma$ we have $-\eta-2\gamma>-\frac{3}{2}>-2$ and therefore $\cB_{-\eta-2\gamma}\hookrightarrow \cB_{-2}$.
        Similar to the example before, it holds that $\cB_{-2}\hookrightarrow H^{-4}(\cO)$. Now we define $F:=\gamma_{\partial} r_\cO \Lambda^{\nu} e_\cO$ in the same way as before, but set $\nu:=-\frac{9}{2}-\tilde{\delta}$ for some $\widetilde{\delta}>0$. In this case $F$ maps $H^{-4}(\cO)$ into $H^{\tilde{\delta}}(\partial \cO)$ and fulfils the Assumption \ref{ass2} 2). Therefore Theorem \ref{local} provides a local-in-time solution in this case. Moreover, since $F$ is linear and bounded, the same holds true for $G=A_{-\sigma} NF$. This means that $F$ satisfies the Assumption \ref{ass2} 3) and Theorem~\ref{thm:global} entails a global-in-time solution

        \end{example}
         %\textcolor{blue}{should we move this part to section 3, so text, assumptions and theorem. and keep here only the scale and the example. and maybe use another notation for the Dirichlet operator, above we have DF...etc for the Fr\'echet derivative.   }

        \begin{remark}
            Before we conclude, %this subsection,
            we compare our results to the ones obtained in~\cite{DuncanMaslowski, DuncanMaslowski2} for {\em additive} infinite-dimensional fractional noise in Hilbert spaces. Similar results have been derived also in Banach spaces in~\cite[Section 5.2]{CMO}. To this aim, let $U$ and $V$ stand for two separable Hilbert spaces. For a $U$-cylindrical fractional Brownian motion $(B^H_t)_{t\geq 0}$ it is known (\cite[Corollary 3.1]{DuncanMaslowski} for $H>\frac{1}{2}$ and \cite[Corollary 11.9]{DuncanMaslowski2} for $H<\frac{1}{2}$) that the stochastic convolution 
            $$\int_0^t S_{t-r}\Phi ~\txtd B^H_r ~~\text{ is well-defined if } ~~\|S(t)\Phi\|_{\cL_2(U,V)}<t^{-\theta},$$
            where $\cL_2(U,V)$ denotes the space of Hilbert-Schmidt operators from $U$ to $V$, $\Phi\in\cL(U,V)$ and $\theta<H$.  Consequently, in order to incorporate boundary noise given by an  $U:=L^2(\partial\cO)$-cylindrical fractional Brownian motion with covariance operator $Q^{\frac{1}{2}}\in \cL_2(U)$, one has to verify that
            \begin{align*}
                \| A S(t) N Q^{\frac{1}{2}}\|_{\cL_2(U,V)}\leq t^{-\theta},
            \end{align*}
            where $V:=L^2(\cO)$ and $\theta<H$. Setting $\cB_\varepsilon:=D(A^\varepsilon)$ as in Section~\ref{main}, one infers that
            \begin{align*}
                \| A S(t) N Q^{\frac{1}{2}}\|_{\cL_2(U,V)}\leq \|A S(t)\|_{\cL(\cB_\varepsilon,V)}\|N\|_{\cL(U,\cB_\varepsilon)}\|Q^{\frac{1}{2}}\|_{\cL_2(U)} \leq c t^{\varepsilon-1} .
            \end{align*}
            This leads to the condition $\varepsilon>1-H$. Since $\varepsilon<\frac{3}{4}$ (recall Corollary~\ref{cor:RoughConvInDomain}), this means that it is possible to deal with fractional Neumann boundary noise for $H>\frac{1}{4}$.
            The condition $\varepsilon>1-H$ is consistent with the result of Corollary~\ref{cor:RoughConvInDomain}, which is established in the more general setting of rough path theory. Moreover, due to the multiplicative structure of the noise, the theory of extrapolation operators is additionally required. 
        \end{remark}
%    \textcolor{blue}{At the end do not forget to check if all the references are indeed cited. If not we remove them or add the citation. Check if the notations are consistent, ex parameter, time-dependency in the r.p., Neumann operator ohne - etc. }

\section*{Conflict of interest statement}
The authors have no conflicts of interest to declare. All co-authors have seen and agree with the contents of the manuscript and there is no financial interest to report. 
\section*{Data availability}
No data was used for the research described in the article.


\begin{thebibliography}{30}
    
    \bibitem{Amann}
    H. Amann.
    \newblock {Nonhomogeneous Linear and Quasilinear Elliptic and Parabolic
    Boundary value problems}.
    \newblock {\em Function spaces, Differential Operators and
    Nonlinear Analysis, Teubner-Texte Math.}, 133:9--126, 1993.
    
    \bibitem{Amann2}
    H. Amann.
    \newblock { Linear and Quasilinear Parabolic Parabolic Problems. Vol. I: Abstract linear Theory}
    \newblock {\em Monogr. Math.}, 89, 1995.
    
    \bibitem{alos}
    E. Al\`os, S. Bonaccorsi.
    \newblock {Stability for stochastic partial differential equations with Dirichlet white-noise boundary conditions}
    \newblock {\em  Infin. Dimens. Anal. Quantum Probab. Relat. Top.} 5(4):125--154, 2002.
    
    \bibitem{Arnold}
    L.~Arnold.
    \newblock{\em Random Dynamical Systems}. Springer, Berlin Heidelberg, Germany, 2003.
    
    \bibitem{BRiedelScheutzow}
    I.~Bailleul, S.~Riedel, M.~Scheutzow.
    \newblock Random dynamical system, rough paths and rough flows.
    \newblock {\em J. Differential. Equat.}, 262(12):5792--5823, 2017.
    
    \bibitem{ns}
    T.~Binz, M.~Hieber, A.~Hussein and M.~Saal.
    \newblock The primitive equations with stochastic wind driven boundary conditions.
    {\em arXiv:2009.09449v2}, 2022.
    
    \bibitem{bt}
    P.~Benner and C.~Trautwein.
    \newblock A Linear Quadratic Control Problem for the Stochastic Heat Equation
    Driven by Q-Wiener Processes.
    \newblock{\em 
    J.~Math.~Anal.~ Appl.},
    457(1):776--802, 2018.
    
    \bibitem{dirk}
    D.~Bl\"omker, B.~Ghayebi and S.~Hosseini.
    \newblock Numerical solution of the Burgers equation with Neumann boundary noise.
    \newblock {\em J. Comput. Appl. Math.}, 311: 148--164, 2017.
    
    \bibitem{BP}
    Z.~Brze\'zniak and S.~Peszat. \newblock Hyperbolic equations with random boundary conditions.
    Recent Development in Stochastic Dynamics and Stochastic Analysis. Vol. 8 of Interdisciplinary Mathematical Sciences, World Scientific, Singapore, 2010.
    
    \bibitem{BDS}
    P.~Brune, J.~Duan and B.~Schmalfu\ss{}.
    \newblock Random Dynamics of the Boussinesq System with Dynamical Boundary Conditions.
    \newblock {\em Stoch.~Anal.~Appl.}, 27(5), 2009.
    
    \bibitem{CastaingValadier}
    C.~Castaing, M.~Valadier. 
    \newblock {\em Convex analysis and measurable multifunctions}.
    \newblock Springer-Verlag, Berlin, 1977. Lecture Notes in Mathematics, Vol. 580.
    
    \bibitem{Cerrai}
    S.~Cerrai and M.~Freidlin.
    \newblock Fast transport asymtotics for stochastic RDEs with boundary noise.
    \newblock {\em Ann.~Probab.}, 39(1):369--405, 2011.
    
     \bibitem{CMO}
     P.~\v{C}oupek, B.~Maslowski and M.~Ondrej\'at.
    \newblock Stochastic integration with respect to fractional processes in Banach spaces.
    \newblock {\em J.~Funct.Anal.~} 282(8):109393, 2022.
    
    \bibitem{daprato}
    G.~Da Prato and J.~Zabczyk. Evolution equations with white-noise boundary conditions. {\em Stoch. Stoch. Rep}. 42:167-182, 1993.
    
    \bibitem{debussche}
    A.~Debussche, M.~Fuhrman and G.~Tessitore.
    \newblock{Optimal control of a stochastic heat equation with boundary-noise and boundary-control}.
    \newblock{\em ESAIM}, 13(1):178--205, 2007.
    
    \bibitem{DDHPV}
    R.~Denk, G.~Dore, M.~Hieber, J.~Pr\"uss and A.~Venni.
    \newblock{New thoughts on old results of R.T. Seeley}.
    \newblock{\em Math. Ann.} 328(4):545--583, 2004.
    \bibitem{DuncanMaslowski}
    T.E.~Duncan, B.~Pasik-Duncan and B.~Maslowski.
    \newblock Fractional Brownian motion and stochastic equations in Hilbert spaces. 
    {\em Stoch. Dyn.}, 2(2):225--250, 2002.
    
    \bibitem{DuncanMaslowski2}
    T.E.~Duncan, B.~Pasik-Duncan and B.~Maslowski.
    \newblock Linear stochastic equations in a Hilbert space with a fractional Brownian motion.
    \newblock Book chapter (pp. 201--221) in  {\em Stochastic Processes, Optimization, and Control Theory: Applications in Financial Engineering, Queueing Networks, and Manufacturing Systems}, 2006.
    
    \bibitem{Stefanie}
    K.~Fellner, S.~Sonner, B.Q.~Tang, D.D~Thuan. \newblock Stabilisation by noise on the boundary for a Chafee-Infante equation with dynamical boundary conditions.
    \newblock {\em Discrete Contin.~Dyn.~Syst.~B}. 24(8): 4055--4078, 2019.
    
    \bibitem{FrizHairer}
    P.K. Friz and M. Hairer. 
    \newblock {\em  A course on rough paths with an introduction to regularity structures}. Second ed., Springer, 2020.

\bibitem{Maria}
 M.J. Garrido-Atienza, K. Lu, B. Schmalfu\ss{}. Random dynamical systems for stochastic
evolution equations driven by multiplicative fractional Brownian noise with Hurst parametes
$H \in (1/3, 1/2]$.
\newblock {\em SIAM J. Appl. Dyn. Syst.}, 15(1):625--654, 2016.
    
    \bibitem{GHairer}
    A.~Gerasimovics and M.~Hairer.
    \newblock H\"ormander's theorem for semilinear SPDEs.
    \newblock {\em Electron. J. Probab.}, 24:1--56, 2019.
    
    \bibitem{GHN}
    A.~Gerasimovics, A.~Hocquet and T.~Nilssen.
    \newblock Non-autonomous rough semilinear PDEs and the
    multiplicative Sewing Lemma.
    \newblock {\em J. Func. Anal.}, 218(10):109200, 2021.
    
    \bibitem{Gubinelli}
    M.~Gubinelli.
    \newblock Controlling rough paths.
    \newblock {\em J. Func. Anal}. 216(1):86--140, 2004.
    
    \bibitem{GubinelliTindel}
    M.~Gubinelli and S.~Tindel.
    \newblock Rough evolution equations.
    \newblock {\em Ann. Probab}. 38(1):1--75, 2010.
    
    \bibitem{HP}
    L.-S.~Hartmann and I.~Pavlyukevich. 
    \newblock Advection-diffusion equation on a half-line with boundary L\'evy noise.
    {\em Discrete.~Contin.~Dyn.~Syst B}, 24(2):637--655, 2019.
    
    \bibitem{HN19} 
    R.~Hesse and A.~Neam\c tu.  \newblock Local mild solutions for rough stochastic partial differential equations.
    \newblock {\em J.~Differential Equat.}, 267(11):6480-6538, 2019.
    
    \bibitem{HN20}
    R.~Hesse and A.~Neam\c tu.
    \newblock Global solutions and random dynamical systems for rough evolution equations.
    \newblock {\em Discrete Contin. Dyn. Syst.}, 25(7):2723--2748, 2020.
    
    \bibitem{HN21}
    R.~Hesse and A.~Neam\c tu. Global solutions for semilinear rough partial differential equations. {\em Stoch. Dyn.}, 22:2240011, 2022. %\url{https://doi.org/10.1142/S0219493722400111}.
    
    \bibitem{Hor}
    L.~H\"ormander.
    \newblock The Analysis of Linear Partial Differential Operators III: Pseudo-Differential Operators.
    \newblock {\em Classics Math.} Reprint of the 1994 edition, Springer, 2007.
    
    \bibitem{Maslowski}
B.~Maslowski.
\newblock Stability of semilinear equations with boundary and pointwise noise.
{\em Annali della Scuola Normale Superiore di Pisa, Classe di Scienze}, 22(1):55--93, 1995.

    \bibitem{m}
    B.~Maslowski and J.~Posp\'i\v{s}il.
    \newblock Parameter estimates for linear partial differential equations with fractional boundary noise.
    \newblock {\em Commun. Inf. Syst.}, 7(1):1--20, 2007.
    
    \bibitem{Munteanu}
    I.~Munteanu.
    \newblock Stabilization of stochastic parabolic equations with boundary-noise and boundary-control.
    \newblock {\em J.~Math.~Anal.~Appl.},  449(1):829--842, 2017.
    
    \bibitem{Triebel78}
    H.~Triebel.
    \newblock{Interpolation theory, function spaces, differential operators} North-Holland Publishing Co., 1978.
    
    \bibitem{VeraarSchnaubelt}
    R. Schaubelt and M.C. Veraar.
    \newblock Stochastic Equations with Boundary Noise.
    \newblock {\em Parabolic Problems, Progr. Nonlinear Differential Equations Appl.}, 80:609--629, 2011.
    
    \bibitem{Wang}
    P.~P.~Wang and C. Zheng.
    \newblock Contaminant transport models under random sources.
    \newblock {\em Ground Water}, 43:423--433, 2005.
    
    \end{thebibliography}
\end{document}